\documentclass{amsart}

\usepackage[pagewise]{lineno} 
\usepackage{srcltx}
\usepackage{xcolor}
\usepackage{amsmath}
\usepackage{amstext}
\usepackage{amssymb}
\usepackage{latexsym}
\usepackage{mathrsfs}
\usepackage{fancybox,color}
\usepackage{enumerate}
\usepackage[latin1]{inputenc}
\usepackage{hyperref}
\usepackage{eurosym}
\usepackage{url}
\usepackage{cite}

\newcommand{\kom}[1]{}
\renewcommand{\kom}[1]{{\bf [#1]}}

 \def\1{\raisebox{2pt}{\rm{$\chi$}}}

\newtheorem{theorem}{Theorem}[section]

\newtheorem{lemma}[theorem]{Lemma}

\newtheorem{definition}[theorem]{Definition}
\newtheorem{remark}[theorem]{Remark}

 \def\1{\raisebox{2pt}{\rm{$\chi$}}}
 
\newcommand{\osc}{\operatornamewithlimits{osc}}

\def\vint_#1{\mathchoice%
          {\mathop{\kern 0.2em\vrule width 0.6em height 0.69678ex depth -0.58065ex
                  \kern -0.8em \intop}\nolimits_{\kern -0.4em#1}}%
          {\mathop{\kern 0.1em\vrule width 0.5em height 0.69678ex depth -0.60387ex
                  \kern -0.6em \intop}\nolimits_{#1}}%
          {\mathop{\kern 0.1em\vrule width 0.5em height 0.69678ex
              depth -0.60387ex
                  \kern -0.6em \intop}\nolimits_{#1}}%
          {\mathop{\kern 0.1em\vrule width 0.5em height 0.69678ex depth -0.60387ex
                  \kern -0.6em \intop}\nolimits_{#1}}}
\def\vintslides_#1{\mathchoice%
          {\mathop{\kern 0.1em\vrule width 0.5em height 0.697ex depth -0.581ex
                  \kern -0.6em \intop}\nolimits_{\kern -0.4em#1}}%
          {\mathop{\kern 0.1em\vrule width 0.3em height 0.697ex depth -0.604ex
                  \kern -0.4em \intop}\nolimits_{#1}}%
          {\mathop{\kern 0.1em\vrule width 0.3em height 0.697ex depth -0.604ex
                  \kern -0.4em \intop}\nolimits_{#1}}%
          {\mathop{\kern 0.1em\vrule width 0.3em height 0.697ex depth -0.604ex
                  \kern -0.4em \intop}\nolimits_{#1}}}

\newcommand{\kint}{\vint}

\newcommand{\aveint}[2]{\mathchoice%
          {\mathop{\kern 0.2em\vrule width 0.6em height 0.69678ex depth -0.58065ex
                  \kern -0.8em \intop}\nolimits_{\kern -0.45em#1}^{#2}}%
          {\mathop{\kern 0.1em\vrule width 0.5em height 0.69678ex depth -0.60387ex
                  \kern -0.6em \intop}\nolimits_{#1}^{#2}}%
          {\mathop{\kern 0.1em\vrule width 0.5em height 0.69678ex depth -0.60387ex
                  \kern -0.6em \intop}\nolimits_{#1}^{#2}}%
          {\mathop{\kern 0.1em\vrule width 0.5em height 0.69678ex depth -0.60387ex
                  \kern -0.6em \intop}\nolimits_{#1}^{#2}}}

\newcommand{\dist}{\operatorname{dist}}

\newcommand{\midrg}{\operatornamewithlimits{midrange}}

\newcommand{\diam}{\operatorname{diam}}

\numberwithin{equation}{section}
\allowdisplaybreaks

\definecolor{color1}{rgb}{0.309, 0.43,0.258}
\definecolor{color2}{rgb}{0.741, 0.502,0.743}
\definecolor{color3}{rgb}{0.580, 0.163,0.107}

\title[] 
      {local Lipschitz regularity for functions satisfying a time-dependent dynamic programming principle}

\author[Jeongmin Han]{}

\subjclass{Primary: 35K20; Secondary: 91A15.}
 \keywords{Parabolic $p$-Laplacian, Lipschitz continuity, Tug-of-war.}

 \email{hanjm9114@snu.ac.kr}

\begin{document}

\maketitle
%
%
%
%
\centerline{\scshape  Jeongmin Han$^*$}
\medskip
{\footnotesize
 \centerline{Department of Mathematical Sciences, Seoul National University,}
   \centerline{Seoul 08826, Republic of Korea}
} 

\bigskip


\begin{abstract}
We prove in this article that functions satisfying a dynamic programming principle have a local interior Lipschitz type regularity. This DPP is partly motivated by the connection to the normalized parabolic $p$-Laplace operator.
\end{abstract}

\section{Introduction}
In this paper, we study functions satisfying the following dynamic programming principle (DPP)
\begin{align}\label{dpp} \begin{split}
& \hspace{-0.2em} u_{\epsilon}(x,t) \\
& \hspace{-0.5em} = \frac{1}{2}\sup_{\nu \in S^{n-1}}  \hspace{-0.15em}\bigg\{ \alpha   u_{\epsilon} \bigg(\hspace{-0.25em}x+\epsilon \nu,t-\frac{\epsilon^2}{2} \bigg) \hspace{-0.15em}+ \hspace{-0.25em}\beta\hspace{-0.1em}  \kint_{B_{\epsilon}^{\nu} }  u_{\epsilon} \bigg(\hspace{-0.15em}x+h,t-\frac{\epsilon^2}{2} \bigg) d \mathcal{L}^{n-1}(h) \hspace{-0.15em} \bigg\} 
\\
& \hspace{-0.5em} + \frac{1}{2}\inf_{\nu \in S^{n-1}}  \hspace{-0.15em}\bigg\{ \alpha   u_{\epsilon} \bigg(\hspace{-0.25em}x+\epsilon \nu,t-\frac{\epsilon^2}{2} \bigg) \hspace{-0.15em}+ \hspace{-0.25em}\beta\hspace{-0.1em}  \kint_{B_{\epsilon}^{\nu} }  u_{\epsilon} \bigg(\hspace{-0.15em}x+h,t-\frac{\epsilon^2}{2} \bigg) d \mathcal{L}^{n-1}(h) \hspace{-0.15em} \bigg\}  \end{split}
\end{align}
for small $\epsilon > 0$.
Here, $\alpha, \beta$ are positive constants with $\alpha + \beta = 1$, $S^{n-1}$ is the $(n-1) $-dimensional unit sphere centered at the origin, $B_{\epsilon}^{\nu} $ is an $(n-1)$-dimensional $\epsilon$-ball which is centered at the origin and orthogonal to a unit vector $\nu$ and 
$$\kint_{A } u(h) d \mathcal{L}^{n-1}(h) = \frac{1}{|A|}\int_{A } u(h) d \mathcal{L}^{n-1}(h) , $$ where $|A|$ is the $(n-1)$-dimensional Lebesgue measure of a set $A$.
We will show interior (parabolic) Lipschitz type regularity for $u_{\epsilon}$ satisfying \eqref{dpp}, that is, 
$$ |u_{\epsilon}(x,t)-u_{\epsilon}(z,s)| \le C( |x-z|+ |t-s|^{\frac{1}{2}} + \epsilon) $$
for some constant $C>0$ and  any $(x,t), (z,s) $ in a parabolic cylinder of a given domain. 

The motivation to study this DPP partly stems from its connection to stochastic games.
On the other hand, our work is also linked to a normalized parabolic $p$-Laplace equation
\begin{equation}
\label{nple} 
\partial_{t} u = \Delta_{p}^{N} u = \Delta u + (p-2) \Delta_{\infty}^{N} u. 
\end{equation}
There have been many recent results regarding mean value characterizations for the $p$-Laplace type equations
(see, for example, \cite{MR2376662,MR2451291,MR2588596,MR2684311,MR2875296,MR2566554,MR3011990}).
We can formally justify that a solution of \eqref{nple} asymptotically satisfies \eqref{dpp} by using the Taylor expansion.

In \cite{MR3494400}, Parviainen and Ruosteenoja proved Lipschitz type regularity for functions satisfying a DPP related to the PDE \eqref{nple}, but they have a different DPP and it only covers the case $2 < p< \infty$.
They also showed H\"{o}lder type estimate for other DPP which is associated with the normalized parabolic $p(x,t)$-Laplace equation.
They used an analytic method in order to show the H\"{o}lder type regularity. 
Meanwhile, for the Lipschitz regularity when $p$ is constant, a core approach in the proof is based on game theory.
The aim of this paper is to extend regularity results in \cite{MR3494400} from the case $2 < p< \infty$ to the case $1 < p< \infty$. 
It is hard to apply the game theoretic argument in that paper to our DPP. 
Therefore here, we extend the proof of H\"{o}lder regularity results in \cite{MR3494400} to obtain the main result, Theorem \ref{mainthm}.

The proof of our main theorem is divided into two parts. 
In the first part, we provide an estimate for the function $u_{\epsilon}$ with respect to $t$.
To be more precise it shows a relation between the oscillation of $u_{\epsilon}$ in time direction and the oscillation in spatial direction. 
Next, we concentrate on proving regularity results with respect to $x$. 
We first obtain H\"{o}lder type estimate and then turn to Lipschitz estimate. 
Comparison arguments play a key role in the proof of the main theorem. 

As we mentioned earlier, our work is closely related to the $p$-Laplace type equations.
The DPP can be understood as a discretization of the related PDE.
Therefore, we can expect that key ideas in studying DPP would be useful in order to analyze the PDE.
On the other hand, our work is in close connection with game theory. 
One can understand the DPP \eqref{dpp} in the spirit of tug-of-war games. 
This interpretation is quite useful in that it allows us to see the problem from a different angle.
Actually, game theoretic arguments have played an important role in proving results in several previous studies.

The notion of a `harmonious function' was introduced in \cite{MR1617186,MR2346452}. 
A harmonious function $v$ satisfies the following DPP
\begin{align} \label{hsdpp} v(x)   =\frac{1}{2} \big\{ \sup_{y \in D} v(y)  +  \inf_{y \in D} v(y) \big\} , \end{align}
where $D$ is a fixed neighborhood of $x$.
In \cite{MR2449057}, some properties of harmonious functions were deduced by using tug-of-war games. 
A relation between the tug-of-war with noise and $p$-Laplace operator was shown in \cite{MR2451291}.
Moreover, similar connections for general fully nonlinear equations were covered in \cite{MR2588596}.
In \cite{MR2684311,MR2566554,MR2875296}, the authors derived asymptotic mean value characterizations for solutions to $p$-Laplace operators.  
The coincidence of game values of tug-of-war games and functions satisfying related DPPs as well as the existence and uniqueness of these functions were shown in \cite{MR3161602}.
Studies on DPPs and associated tug-of-war games are ongoing under various settings, for example in nonlocal and Heisenberg group setting, as in \cite{MR2471139, MR2868849, MR3177660}. 

Many regularity results are also known for functions defined through a DPP.  
In \cite{MR3011990}, a Lipschitz type estimate was proved for a DPP connected to the elliptic $p$-Laplace problem.
A local approach for the regularity was developed in \cite{MR3169768} (see also \cite{MR3441079}).
It is based on cancellation strategies which as an application give a new and straightforward proof for the Lipschitz continuity for the corresponding PDEs.
On the other hand, in \cite{MR3623556}, interior H\"{o}lder regularity was shown for a space-varying DPP based on the method in \cite{MR3846232}. 
Lipschitz regularity for this DPP was proved in \cite{alpr2018lipschitz}.

The paper is organized as follows. In the next section, some notations and background knowledge are presented. 
We prove the main theorem in the remaining sections. 
In Section 3, we establish the estimate for our function $u_{\epsilon}$ with respect to $t$. 
After that, regularity for $u_{\epsilon}$ in spatial direction is covered. 
We derive the H\"{o}lder regularity in Section 4 and the Lipschitz regularity in Section 5.

\section{Preliminaries}

Fix $ n \ge 2 $ and let $ \Omega \subset \mathbb{R}^{n} $ be a bounded domain. 
We consider a parabolic cylinder $\Omega_{T} := \Omega \times (0, T]$ for $T>0$ and its parabolic boundary $$ \partial_{p}\Omega_{T} = ( \partial \Omega  \times  [0,T] ) \cup ( \Omega \times \{ 0 \} ) .$$ 
Let $\epsilon > 0$ and define a parabolic $\epsilon$-strip of $\Omega_{T}$ as follows:
$$ \Gamma_{\epsilon, T} = \bigg( \Gamma_{\epsilon} \times \bigg( -\frac{\epsilon^{2}}{2},T  \bigg] \bigg)  \cup \bigg( \Omega \times \bigg( -\frac{\epsilon^{2}}{2},0 \bigg] \bigg) ,$$
where $\Gamma_{\epsilon} = \{ x \in \mathbb{R}^{n} \backslash 
 \Omega : \dist (x, \partial \Omega) \le \epsilon \}$ is an $\epsilon$-strip of $\Omega$.
Let $F$ be a given function defined in $\Gamma_{\epsilon, T} $.
\begin{definition}
Let  $\alpha, \beta \in (0,1)$ with $\alpha + \beta = 1$.
We say that a function $u_{\epsilon}$ satisfies the $\alpha$-parabolic DPP (with boundary data $F$) if
\eqref{dpp} holds in $\Omega_{T}$ and $u_{\epsilon}=F$ in $\Gamma_{\epsilon, T} $. 
\end{definition}
Here, we remark that if the boundary data $F$ is bounded, then one can show that $u_{\epsilon}$ is also bounded since $ ||u_{\epsilon}||_{L^{\infty}(\Omega_{ T})} \le ||F||_{L^{\infty}(\Gamma_{\epsilon,T})}$. (cf. \cite{MR3011990,MR3494400 })

We can heuristically interpret these functions in terms of  `time-dependent tug-of-war game with noise'. 
This game is a two player zero-sum game in $ \Omega_{T} $. 
The procedure of the game is as follows. When the game is started, a token is located at some point $ (x_{0}, t_{0}) \in \Omega_{T} $. First Player I and Player II choose some directions  $\nu_{I}, \nu_{II} $ in the $(n-1) $-dimensional unit sphere centered at the origin $S^{n-1}$, respectively. Next one tosses a fair coin and the winner of the toss moves the token.  With probability $\alpha$, the winner Player $i(\in \{ I, II \})$ moves the token to the point 
$ x_{1} = x_{0}+ \epsilon \nu_{i} \in B_{\epsilon}(x_{0}) $ and simultaneously the time changes by $ t_{1}=t_{0}- \epsilon^{2}/2$. 
On the other hand, with probability $\beta$, the token will be moved to the point $x_{1} $ where $x_{1}$ is randomly chosen from the uniformly probability distribution on the $(n-1)$-dimensional $\epsilon$-ball $B_{\epsilon}^{\nu_{i}}(x_{0})$ which is centered at $x_{0}$ and is orthogonal to $\nu_{i}$ and simultaneously the time also changes by $t_{1}=t_{0}-\epsilon^{2}/2$.
 If $(x_{1}, t_{1}) \in  \Gamma_{\epsilon, T} $, the game ends and Player II pays Player I the payoff $F(x_{1},t_{1})$. Otherwise, the above process is repeated and the token is moved to a point $ (x_{2}, t_{2}) \in B_{\epsilon}(x_{1})  \times \{  t_{1}  - \epsilon^{2}/2 \}$. The game ends when the token is located in the parabolic strip $\Gamma_{\epsilon, T}$ for the first time. Since $ t_{k} = t_{0} - k \epsilon^{2}/2 < 0 $ for sufficiently large $  k$, the game must terminate in finite time. 

Let $(x_{\tau}, t_{\tau}) $ be the end point of the game. We are concerned with the expectation of the payoff $ F(x_{\tau}, t_{\tau}) $. Player I tries to maximize $ F(x_{\tau}, t_{\tau}) $ and Player II tries to minimize that. 
The value function of Player I and II are defined as  
$$ u_{\epsilon}^{I}(x_{0},t_{0}) = \sup_{\mathcal{S}_{I}} \inf_{\mathcal{S}_{II}} \mathbb{E}_{\mathcal{S}_{I},\mathcal{S}_{II}}^{(x_{0},t_{0})}[F(x_{\tau},t_{\tau})]$$
and
$$ u_{\epsilon}^{II}(x_{0},t_{0}) = \inf_{\mathcal{S}_{II}}  \sup_{\mathcal{S}_{I}} \mathbb{E}_{\mathcal{S}_{I},\mathcal{S}_{II}}^{(x_{0},t_{0})}[F(x_{\tau},t_{\tau})] ,$$ where $ \mathcal{S}_{I}$ and $\mathcal{S}_{II} $ are strategies for Player I and Player II, respectively.

By the definition of the game, we can make a rough guess that $ u_{\epsilon}^{I}$ and $ u_{\epsilon}^{II}$ satisfy \eqref{dpp} since the value of these functions at every point would coincide the expectation value of it in the next turn. 
Although we will not show the relation between $ u_{\epsilon}^{I}$, $ u_{\epsilon}^{II}$ and $u_{\epsilon}$ in this paper, the above description of the game gives some intuition in the proof of our main result, Theorem \ref{mainthm}. 

We will use the notation $B_{\epsilon}^{\nu} $ and $S^{n-1}$ as in the previous section. 
Let $r$ be a fixed positive number. 
For $a>0$, set
$$ Q_{ar} = B_{ar}(0) \times (-ar^{2}, 0),$$ 
$$ Q_{ar, \epsilon} = B_{ar + \epsilon}(0) \times (-ar^{2}-\epsilon^{2}/2 , 0) $$
and
 $$ \Sigma_{a} = \{ (x,z,t,s) : x,z \in B_{ar}(0), -ar^{2}< t < 0 , |t-s|<\epsilon^{2}/2  \} $$
and we write $\Lambda_{t, \epsilon}$ for a $\epsilon$-time slice $$ \Lambda_{t, \epsilon} = B_{r+\epsilon}(0) \times (t-\epsilon^{2}/{2}, t] .$$
Furthermore, let
 $$ \midrg_{i \in I}A_{i}= \frac{1}{2} \bigg( \sup_{i \in I}A_{i} +\inf_{i \in I}A_{i} \bigg) $$ and
$$ \mathscr{A}u_{\epsilon} (x, \nu,t) = \alpha u_{\epsilon}(x+ \epsilon \nu,t) + \beta \kint_{B_{\epsilon}^{\nu} }u_{\epsilon}(x+ h,t) d \mathcal{L}^{n-1}(h) , $$
where $\nu \in S^{n-1} $ and $\kint_{A}$ means average of the integration on a set $A$.
Then we can rewrite (\ref{dpp}) by
\begin{align} \label{dppref} u_{\epsilon}(x,t)  = \midrg_{\nu \in S^{n-1}} \mathscr{A}u_{\epsilon} \bigg(x, \nu,t-\frac{\epsilon^2}{2} \bigg). \end{align}
We also define a set $ \mathbf{R}_{\nu}$ such that
$$ \mathbf{R}_{\nu} = \{ M \in \mathbf{O}(n) : Me_{1} = \nu \} ,$$
where $\mathbf{O}(n) $ is the orthogonal group in dimension $n$ and $e_{1} $ is the first vector in the standard orthonormal basis.
For simplicity, we abbreviate $$\sup_{\substack{\nu_{x},\nu_{z} \in  S^{n-1} \\ (P_{\nu_{x}}, P_{\nu_{z}}) \in \mathbf{R}_{\nu_{x}} \times \mathbf{R}_{\nu_{z}}} } $$ to
$$\sup_{\nu_{x},\nu_{z} \in  S^{n-1} } $$ throughout the paper.

\section{Regularity with respect to time}
First we investigate regularity for the function $u_{\epsilon}$ with respect to $t$. 
The aim of this section is to prove Lemma \ref{lem2} below. 
This lemma provides some information about a relation between the oscillation in a time slice and that in the whole cylinder. 

We use a comparison argument in the proof of the lemma.
We will first find an appropriate function $ \bar{v}$ ($\underline{v}$, respectively) which plays a similar role as a supersolution (subsolution, respectively) in PDE theory.
After that, we will deduce the desired result by estimating the difference of those functions.
The method used here is motivated by \cite[Lemma 4.3]{MR3660769}. Our proof may be regarded as a discrete version of this lemma. 

From now on, we fix $0<r<1$ and $T>0$. 
Since we only consider interior regularity, it is sufficient to show the regularity result in a cylinder $Q_{r}= B_{r}(0) \times (-r^{2}, 0)$ with proper translation.
We still use the notation $ \Omega_{T}$ after the translation.

\begin{lemma}\label{lem2}
Let $\bar{Q}_{2r} \subset \Omega_{T}$, $ -r^{2} < s < t < 0  $ and $u_{\epsilon}$ satisfies the $\alpha$-parabolic DPP with boundary data $F \in L^{\infty}(\Gamma_{\epsilon, T} )$ for given $0< \alpha<1$. Then, for given $\epsilon > 0$, $u_{\epsilon}$ satisfies the estimate
$$ |u_{\epsilon} (x,t) - u_{\epsilon} (x,s) | \le 18  \sup_{-r^{2} <\tau<0}  \osc_{\Lambda_{\tau, \epsilon}} u_{\epsilon} $$
for any $x \in B_{r}$.  
\end{lemma}

\begin{proof} 
We set
$$ A = \sup_{-r^{2} <\tau<0}  \osc_{\Lambda_{\tau, \epsilon}} u_{\epsilon} $$ and
$$ \bar{v}_{c}(x,t) = c + 7 r^{-2} A t + 2 r^{-2} A |x|^{2}, $$
where $c \in \mathbb{R}$.
Define
$$ \bar{c} = \inf \{ c \in \mathbb{R} : \bar{v}_{c}  \ge u_{\epsilon} \  \textrm{in} \  \Lambda_{-r^{2}, \epsilon} \} $$ and we write $ \bar{v} = \bar{v}_{\bar{c}} $. 
Then for any $\eta >0$, we can always choose $ (x_{\eta}, t_{\eta}) \in \Lambda_{-r^{2}, \epsilon} $ so that $$ u_{\epsilon} (x_{\eta}, t_{\eta}) \ge \bar{v} (x_{\eta}, t_{\eta}) - \eta .$$ 
In this case, there would be some accumulation points $(\bar{x}, \bar{t}) \in \bar{\Lambda}_{-r^{2}, \epsilon} $ as $ \eta \to 0 $. Furthermore, $ \bar{x} $ must satisfy $ |\bar{x}| \le r $, since if not,
$$ 2A \le \bar{v}(x_{\eta}, t_{\eta}) - \bar{v}(0, t_{\eta}) \le u_{\epsilon}(x_{\eta}, t_{\eta}) - u_{\epsilon}(0, t_{\eta})+\eta \le A+\eta $$ for any $\eta > 0$, then it is a contradiction when $A > 0$.

Now we compare $  \midrg_{\nu \in S^{n-1}} \mathscr{A} \bar{v} (x, \nu,t-\epsilon^{2}/2 )  $ with $\bar{v}(x,t) $.
First, observe that
\begin{align*}& \midrg_{\nu \in S^{n-1}} \mathscr{A} \bar{v} \bigg(x, \nu,t-\frac{\epsilon^2}{2} \bigg)  \\ & \hspace{-0.15em} \le \alpha \midrg_{\nu \in S^{n-1}} \bar{v} \bigg(x\hspace{-0.15em}+\hspace{-0.15em} \epsilon \nu,t-\frac{\epsilon^2}{2} \bigg) \hspace{-0.15em} + \hspace{-0.15em} \beta \hspace{-0.35em} \sup_{\nu \in S^{n-1}} \hspace{-0.15em} \kint_{B_{\epsilon}^{ e_{1} } }\bar{v}\bigg(x+ P_{\nu} h,t-\frac{\epsilon^{2}}{2} \bigg) d \mathcal{L}^{n-1}(h) 
\end{align*} for some $P_{\nu} \in \mathbf{R}_{\nu} $. 
We see that
\begin{align*} \kint_{B_{\epsilon}^{e_{1}}} |x+ P_{\nu}h|^{2}  d\mathcal{L}^{n-1}(h)   &= \kint_{B_{\epsilon}^{\nu}} (|x|^{2} + 2 \langle x ,  P_{\nu} h \rangle + | P_{\nu}h|^{2} ) d\mathcal{L}^{n-1}(h)  \\ & \le |x|^{2} + \epsilon^{2} 
\end{align*} for any $ \nu \in S^{n-1} $. Next we need to show that
\begin{align*}  \midrg_{\nu \in S^{n-1}} |x+ \epsilon \nu |^{2}  \le |x|^{2}+\epsilon^{2}.
\end{align*} Observe that
\begin{align*} \sup_{\kappa \in B_{\epsilon}} |x+\kappa|^{2} &=\sup_{\nu \in S^{n-1}} \sup_{-\epsilon \le a \le \epsilon} |x+ a \nu|^{2} \\&=\sup_{\nu \in S^{n-1}} \sup_{-\epsilon \le a \le \epsilon} ( a^{2} + 2a\langle x, \nu \rangle + |x|^{2} ).
\end{align*} Since $a^{2} + 2a\langle x, \nu \rangle + |x|^{2} $ is convex in $a$, we observe that
$$ \sup_{-\epsilon \le a \le \epsilon} ( a^{2} + 2a\langle x, \nu \rangle + |x|^{2} ) = \epsilon^{2} + 2\epsilon |\langle x, \nu \rangle| + |x|^{2}.$$ We also see that there is a unit vector $ \mu $ 
so that $$\sup_{\nu \in S^{n-1}}( \epsilon^{2} + 2\epsilon |\langle x, \nu \rangle| + |x|^{2}) = |x+ \epsilon \mu |^{2}, $$ as $S^{n-1}$ is compact. Then we get 
\begin{align*} \midrg_{\nu \in S^{n-1}} |x+ \epsilon \nu |^{2}  \le \frac{1}{2}(|x+ \epsilon \mu |^{2}+|x- \epsilon \mu |^{2}) = |x|^{2}+ \epsilon^{2}.
\end{align*}
Therefore, we discover
\begin{align*}& \midrg_{\nu \in S^{n-1}} \mathscr{A} \bar{v} \bigg(x, \nu,t-\frac{\epsilon^2}{2} \bigg) \\
& \ \le \bar{c} + 7 r^{-2}A \bigg( t - \frac{\epsilon^{2}}{2} \bigg)+ 2r^{-2}A \{  \alpha(|x|^{2}+\epsilon^{2}) + \beta (|x|^{2}+\epsilon^{2}) \} \\
& \ \le  \bar{c} + 7r^{-2}A t + 2r^{-2} A |x|^{2} - \frac{3}{2} r^{-2} A \epsilon^{2} = \bar{v}(x,t) - \frac{3}{2} r^{-2} A \epsilon^{2} .
\end{align*}
Thus, 
\begin{align} \label{ie1}  \midrg_{\nu \in S^{n-1}}\mathscr{A} \bar{v} \bigg(x, \nu,t-\frac{\epsilon^2}{2} \bigg) \le \bar{v} (x,t) \end{align} for all $ (x, t) \in Q_{r} $.

Let $ M = \sup_{Q_{r, \epsilon} \backslash \Lambda_{-r^{2}, \epsilon}}(u_{\epsilon} - \bar{v})  $ and suppose $M >0$.  In this case, we see that $u_{\epsilon} \le \bar{v} + M $ in $Q_{r, \epsilon} $.
For any $\eta ' > 0$, we can choose a point   $ (x_{\eta '}, t_{\eta '}) \in Q_{r, \epsilon}  $ such that 
$$ u_{\epsilon}(x_{\eta'}, t_{\eta'}) > \bar{v}(x_{\eta'}, t_{\eta'}) + M - \eta' .$$
We have to show that $ (x_{\eta '}, t_{\eta '})$ must be in $\bar{Q}_{r} $ for any sufficiently small $\eta' > 0 $. 
By the definition of $M$, $t_{\eta'}> -r^{2}$. 
Note that we cannot assert this when $M \le 0$. 
On the other hand, for any $|x| \ge r $,
$$ \bar{v}(x, t)- \bar{v}(0, t) \ge 2A . $$
We also observe that $ u_{\epsilon}(x,t)-u_{\epsilon}(0,t)  \le A$. Hence it is always true that
$$ (u_{\epsilon}-\bar{v})(x,t) \le (u_{\epsilon}-\bar{v})(0,t) .$$
Thus, $(x_{\eta '}, t_{\eta '}) \in  \bar{Q}_{r}$. Then we obtain that
\begin{align*}  \midrg_{\nu \in S^{n-1}} \mathscr{A} \bigg\{ \bar{v} \bigg(    x_{\eta'}, \nu,t_{\eta'}-\frac{\epsilon^2}{2} \bigg) + M \bigg\} &  \ge  \midrg_{\nu \in S^{n-1}} \mathscr{A} u_{\epsilon} \bigg(x_{\eta'}, \nu,t_{\eta'}-\frac{\epsilon^2}{2} \bigg) \\
& = u_{\epsilon}(x_{\eta'},t_{\eta'}) \\
& > \bar{v}(x_{\eta'},t_{\eta'}) + M -\eta' .
\end{align*} In the first inequality, we have used that $ \bar{v}+M \ge u_{\epsilon} $ in $ Q_{r, \epsilon}$.
Therefore, \begin{align} \label{ie2}  \midrg_{\nu \in S^{n-1}} \mathscr{A} \bar{v} \bigg(x_{\eta'}, \nu,t_{\eta'}-\frac{\epsilon^2}{2} \bigg)  >\bar{v}(x_{\eta'},t_{\eta'}) -\eta' \end{align}
for any $\eta' > 0$. 
We combine (\ref{ie1}) with (\ref{ie2}) to discover that $A = 0$, and so $ \bar{v} = u_{\epsilon}= \bar{c} $.  
If $u_{\epsilon}$ is not a constant function, then we have a contradiction to $ A > 0 $. Hence $ M \le 0$ and therefore $u_{\epsilon} \le \bar{v}$ in $Q_{r, \epsilon} $.

On the other hand, consider 
$$  \underline{v}(x,t) = \underline{c} - 7 r^{-2} A t - 2 r^{-2} A |x|^{2},$$ where 
$$ \underline{c} = \sup \{ c \in \mathbb{R} : \underline{v}_{c}  \le u_{\epsilon} \  \textrm{in} \  \Lambda_{-r^{2}, \epsilon} \}. $$  
Following the above procedure, we can show that $u_{\epsilon} \ge \underline{v}$ in $Q_{r, \epsilon} $. For arbitrary $\eta>0$, we can choose  $(\bar{x}_{\eta}, \bar{t}_{\eta}), (\underline{x}_{\eta}, \underline{t}_{\eta}) \in \bar{\Lambda}_{-r^{2}, \epsilon}$ such that
$$u_{\epsilon} (\bar{x}_{\eta}, \bar{t}_{\eta}) \ge \bar{v} (\bar{x}_{\eta}, \bar{t}_{\eta}) - \eta $$ and $$ u_{\epsilon} (\underline{x}_{\eta}, \underline{t}_{\eta}) \le \bar{v} (\underline{x}_{\eta}, \underline{t}_{\eta}) +\eta.$$
Then 
$$ \bar{v}(\bar{x}_{\eta}, \bar{t}_{\eta})-\underline{v}(\underline{x}_{\eta}, \underline{t}_{\eta}) \le \osc_{\Lambda_{t, \epsilon}} u_{\epsilon} + 2 \eta , $$
and hence
$$ \bar{c}- \underline{c} \le 3A+ \frac{7}{2}r^{-2}A \epsilon^{2} \le 7A. $$
Therefore, we obtain
$$ \osc_{Q_{r}}u_{\epsilon} \le \sup_{Q_{r}} \bar{v}- \inf_{Q_{r}} \underline{v} \le \bar{c} - \underline{c}+7A+4A \le 18A. $$ 
This completes the proof.
\end{proof}

\begin{remark} \label{remark} We showed in the proof of Lemma \ref{lem2} that the oscillation of $u_{\epsilon}$ in time direction is uniformly estimated by the oscillation of $u_{\epsilon}$ in spatial direction on $(\epsilon^{2}/2)$-time slices. 
Note that an $(\epsilon^{2}/2)$-time slice $\Lambda_{t, \epsilon}$ shrinks to $B_{r} \times \{ t \}$ as $\epsilon \to 0$ for any $t$.
Thus, we can see that regularity for $u_{\epsilon}$ with respect to $t$ almost depends on the regularity with respect to $x$ provided $\epsilon$ is small enough. 
\end{remark}

\section{H\"{o}lder regularity}
The aim of this section is to show that $u_{\epsilon}$ satisfies H\"{o}lder type regularity.
This result will be essentially used to prove Lipschitz regularity with respect to $x$ in the next section.

We will use a comparison argument arising from game interpretations for obtaining regularity results in spatial direction. 
This argument plays an important role in obtaining the desired estimate. 
Several regularity results for functions satisfying various time-independent DPPs were proved by calculations based on this argument (see \cite{MR3846232,MR3623556,alpr2018lipschitz}).
It was proved in \cite{MR3494400} that functions satisfying another time-dependent DPP have H\"{o}lder regularity.
Our proof differs from that in \cite{MR3494400} due to the difference of the setting of DPP.

Our argument depends on the distance between two points. 
If two points are relatively far away, we will consider `multidimensional DPP'(For a more detailed explanation, see \cite{MR3846232}).
We divide the argument into two subcases. 
For each case, we will get the desired estimate by choosing proper behavior of an auxiliary function.
In addition, we can derive our estimate by direct calculation when two points are close enough.

\begin{lemma} \label{lem1}
Let $\bar{B}_{2r}(0) \times [-2r^{2}-\epsilon^{2} / 2, \epsilon^{2}/2] \subset \Omega_{T}$, $ 0< \alpha < 1 $ and $ \epsilon > 0 $ is small. Suppose that $u_{\epsilon}$ satisfies the $\alpha$-parabolic DPP with boundary data $F \in L^{\infty}(\Gamma_{\epsilon, T} )$.
Then for any $0 < \delta < 1 $,
$$ |u_{\epsilon} (x,t) - u_{\epsilon} (z,s) | \le C ||u_{\epsilon}||_{\infty} ( |x-z|^{\delta} + \epsilon^{\delta}) ,$$ whenever $x, z \in B_{r}(0)$, $ -r^{2}<t<0$, $ |t-s| < \epsilon^{2} / 2 $ and $C>0$ is a constant which only depends on $r, \delta, \alpha$ and $n$.
\end{lemma}

\begin{proof}
First, we can assume that $ ||u_{\epsilon}||_{\infty} \le r^{\delta} $ by scaling.
Let us construct an auxiliary function. 
Define \begin{align} \label{f1} f_{1}(x,z) = C | x - z |^{\delta} + M | x+ z|^{2} ,\end{align} 
\begin{equation}   \label{f2} f_{2}(x, z) = \left\{ \begin{array}{ll}
C^{2(N-i)} \epsilon ^ {\delta}  & \textrm{if $(x, z) \in A_{i}$}\\
0 & \textrm{if $|x-z|>N \epsilon / 10 $}
\end{array} \right.
\end{equation} and
\begin{align} \label{gg} g(t, s) = \max \{ M ( |t- r^{2}|^ { {\delta}/2} - r^{\delta}),  M ( |s- r^{2}|^ {{\delta}/2} - r^{\delta})   \}\end{align}
 where  $N = N( r, \delta, \alpha,n) \in \mathbb N $,  $ C = C(r, \delta, \alpha, n) >1$ and $M = M(r) > 1$ are constants to be determined, and $$A_{i} = \{ (x , z) \in \mathbb R ^ {2n} : (i-1) \epsilon / 10 < |x-z| \le i \epsilon / 10 \} $$ for $ i = 0, 1, ... , N$.  

Now we define \begin{align} \label{comfn} H(x, z, t, s) = f_{1}(x,z) - f_{2}(x,z) + g(t, s) . \end{align}

We first show that 
$$ |u_{\epsilon}(x,t) - u_{\epsilon}(z,s) | \le  C ( |x-z|^ {\delta}+ \epsilon^ {\delta} ) $$ for every $ x,z \ (x \neq z) \in B_{2r}  (0)$, $ -2r^{2}<t<0$ and $ |t-s| < \epsilon^{2} / 2 $.
To this end, choose $M$ sufficiently large so that $$ u_{\epsilon}(x,t) - u_{\epsilon}(z,s) - H(x,z,t,s) \le  C ^ {2N} \epsilon ^ {\delta}+  C \epsilon^ {\delta} \qquad  \textrm{in} \ \ \ \Sigma_{2} \backslash \Sigma_{1}  .  $$ 
So, if we prove that
\begin{align}  \label{kees}  u_{\epsilon}(x,t) - u_{\epsilon}(z,s) - H (x,z,t,s) \le  C ^ {2N} \epsilon^ {\delta} +C \epsilon^ {\delta}  \qquad  \textrm{in} \ \ \ \Sigma_{1} \backslash \Upsilon
\end{align}
where $ \Upsilon = \{ (x,z,t,s) \in \mathbb{R}^{2n} \times \mathbb{R}^{2} : x   =   z, \ -r^{2}< t < 0 , \ |t-s|<\epsilon^{2}/2 \}$,
then it is shown that Lemma \ref{lem1} holds in $ \Sigma_{2} \backslash \Upsilon $.
Since we can obtain this estimate for $ u_{\epsilon}(z,s) - u_{\epsilon}(x,t) $, we have
$$ | u_{\epsilon}(x,t) - u_{\epsilon}(z,s) | \le  C ^ {2N} \epsilon^ {\delta} +C \epsilon^ {\delta}  + H (x,z,t,s) \qquad  \textrm{in} \ \ \ \Sigma_{2} \backslash \Upsilon .$$
Now we can assume that $z= -x $ by proper scaling and transformation,
and then we get
$$ |u_{\epsilon}(x,t) - u_{\epsilon}(-x,s) | \le     C|x|^{\delta}+  C' \epsilon^ {\delta} $$
for some universal constant $C'>0$.
It gives the result of Lemma \ref{lem1}.

Suppose  that \eqref{kees} is not true.  Then \begin{align} \label{ca1} K := \sup_{(x,z,t,s) \in \Sigma_{1} \backslash \Upsilon } (u_{\epsilon}(x,t) - u_{\epsilon}(z,s) - H(x,z,t,s)) > C ^ {2N} \epsilon ^{\delta}+C \epsilon ^{\delta}  . \end{align}
Let $\eta >0$. We can choose $(x', z', t', s') \in \Sigma_{1} \backslash \Upsilon $ such that 
$$ u_{\epsilon}(x',t') - u_{\epsilon}(z',s') -H(x',z',t',s') \ge K - \eta . $$

Recall the DPP (\ref{dppref}).
Using this together with the previous inequality, we know that
\begin{align*} K & \le  u_{\epsilon}(x',t') - u_{\epsilon}(z',s') -H(x',z',t',s') + \eta
\\ &
 \\ & \le \frac{1}{2} \bigg[ \sup_{\nu_{x'},\nu_{z'} \in S^{n-1}} \bigg\{ \mathscr{A}u_{\epsilon} \bigg( x', \nu_{x'}, t'- \frac{\epsilon^{2}}{2} \bigg) - \mathscr{A}u_{\epsilon} \bigg( z', \nu_{z'}, s'- \frac{\epsilon^{2}}{2} \bigg) \bigg\} \\ & \qquad + \inf_{\nu_{x'},\nu_{z'} \in S^{n-1}} \bigg\{ \mathscr{A}u_{\epsilon} \bigg( x', \nu_{x'}, t'- \frac{\epsilon^{2}}{2} \bigg) - \mathscr{A}u_{\epsilon} \bigg( z', \nu_{z'}, s'- \frac{\epsilon^{2}}{2} \bigg) \bigg\}  \bigg]  \\  
& \qquad - H(x', z', t', s') + 2\eta. \end{align*}

Let 
$$ \mathbf{[I]} = \frac{1}{2}  \sup_{\nu_{x'},\nu_{z'} \in S^{n-1}} \bigg\{ \mathscr{A}u_{\epsilon} \bigg( x', \nu_{x'}, t'- \frac{\epsilon^{2}}{2} \bigg) - \mathscr{A}u_{\epsilon} \bigg( z', \nu_{z'}, s'- \frac{\epsilon^{2}}{2} \bigg) \bigg\}$$
and
$$ \mathbf{[II]} = \frac{1}{2}  \inf_{\nu_{x'},\nu_{z'} \in S^{n-1}} \bigg\{ \mathscr{A}u_{\epsilon} \bigg( x', \nu_{x'}, t'- \frac{\epsilon^{2}}{2} \bigg) - \mathscr{A}u_{\epsilon} \bigg( z', \nu_{z'}, s'- \frac{\epsilon^{2}}{2} \bigg) \bigg\}.$$

We see that
\begin{align*}\  u_{\epsilon}&(x',t')  - u_{\epsilon}(z',s') \\ & =  \midrg_{\nu_{x'} \in S^{n-1}} \mathscr{A}u_{\epsilon} \bigg(x', \nu_{x'},t'-\frac{\epsilon^2}{2} \bigg)   -  \midrg_{\nu_{z'} \in S^{n-1}} \mathscr{A}u_{\epsilon} \bigg(z', \nu_{z'},s'-\frac{\epsilon^2}{2} \bigg)  \\ &  \le  \mathbf{[I]} +  \mathbf{[II]} + \eta .
\end{align*} 

By the definition of $\mathscr{A}$, we see that
\begin{align*}
 & \mathbf{[I]} =  \frac{1}{2}  \sup_{\nu_{x'},\nu_{z'} \in S^{n-1}} \bigg[ \alpha \bigg\{ u_{\epsilon} \bigg( x+\epsilon \nu_{x'}, t'- \frac{\epsilon^{2}}{2} \bigg) - u_{\epsilon} \bigg( z'+ \epsilon \nu_{z'},s'- \frac{\epsilon^{2}}{2} \bigg) \bigg\} \\
& + \beta \kint_{B_{\epsilon}^{\nu_{x'}} } \bigg\{ u_{\epsilon}\bigg(x'+ P_{\nu_{x'}} h,t'-\frac{\epsilon^{2}}{2} \bigg) - u_{\epsilon}\bigg(z'+ P_{\nu_{z'}} h,s'-\frac{\epsilon^{2}}{2} \bigg)  \bigg\} d \mathcal{L}^{n-1}(h)   \bigg].
\end{align*}

Now we estimate $  \mathbf{[I]}$(and $  \mathbf{[II]}$) by $H$-related terms. Let
$$ \mathbf{[III]} = \alpha H(x+ \epsilon\nu_{x},z+ \epsilon \nu_{z}, t,s) + \beta \kint_{B_{\epsilon}^{e_{1}} }H(x+ P_{\nu_{x}}h,z+P_{\nu_{z}} h,t,s) d \mathcal{L}^{n-1}(h) . $$
Recall $ f(x,z)=f_{1}(x,z)- f_{2}(x,z) $ and $ H(x,z,t,s)=f(x,z)+g(t,s) $. Then we see that
$$H(x+ \epsilon\nu_{x},z+ \epsilon \nu_{z}, t,s)=   f(x+ \epsilon\nu_{x},z+ \epsilon \nu_{z}) + g(t,s)$$
and \begin{align*}  \kint_{B_{\epsilon}^{e_{1}} }&H(x+ P_{\nu_{x}}h,z+P_{\nu_{z}} h,t,s) d\mathcal{L}^{n-1}(h) \\ & =\kint_{B_{\epsilon}^{e_{1}} } \big\{ f(x+ P_{\nu_{x}}h,z+P_{\nu_{z}} h) + g(t,s) \big\} d \mathcal{L}^{n-1}(h)
\\ & = \kint_{B_{\epsilon}^{e_{1}} }f(x+ P_{\nu_{x}}h,z+P_{\nu_{z}} h) d \mathcal{L}^{n-1}(h) + g(t,s). 
 \end{align*}
Then we can write $ \mathbf{[III]}$ as
$$  \alpha f(x+ \epsilon\nu_{x},z+ \epsilon \nu_{z}) + \beta \kint_{B_{\epsilon}^{e_{1}} }f(x+ P_{\nu_{x}}h,z+P_{\nu_{z}} h) d \mathcal{L}^{n-1}(h) + g(t,s) . $$
Here we define an operator $T$  as
\begin{align*} Tf&( x, z, P_{\nu_{x}} , P_{\nu_{z}}) \\ & = \alpha f(x+ \epsilon \nu_{x},z+ \epsilon \nu_{z}) + \beta \kint_{B_{\epsilon}^{e_{1}} }f(x+P_{\nu_{x}} h,z+P_{\nu_{z}} h) d \mathcal{L}^{n-1}(h).\end{align*}

Since \begin{align} \label{ineq} u_{\epsilon}(y,t) - u_{\epsilon}(\tilde{y},\tilde{t}) \le K + H(y,\tilde{y},t,\tilde{t}) = K + f(y, \tilde{y}) + g(t,\tilde{t}) \end{align} by the definition of $K$, we obtain that
\begin{align*} \mathbf{[I]}& \le \frac{1}{2} \sup_{\nu_{x'},\nu_{z'} \in  S^{n-1} } \bigg[ \alpha \bigg\{ K+ H \bigg( x'+ \epsilon \nu_{x'}, z'+ \epsilon \nu_{z'}, t'-\frac{\epsilon^{2}}{2}, s'-\frac{\epsilon^{2}}{2} \bigg) \bigg\}
\\ &  + \beta \kint_{B_{\epsilon}^{e_{1}} }  \bigg\{ K +  H \bigg( x'+ P_{\nu_{x'}}h, z'+ P_{\nu_{z'}}h, t'-\frac{\epsilon^{2}}{2}, s'-\frac{\epsilon^{2}}{2} \bigg)   \bigg\} d \mathcal{L}^{n-1}(h) \bigg]
\\ & \le \frac{1}{2} \bigg[ K +\sup_{\nu_{x'},\nu_{z'} \in  S^{n-1} } Tf ( x', z',P_{\nu_{x'}} , P_{\nu_{z'}}  ) + g\bigg( t'-\frac{\epsilon^{2}}{2}, s'-\frac{\epsilon^{2}}{2} \bigg) \bigg]. \end{align*}

Next we have to estimate $ \mathbf{[II]} $. Choose $ \rho_{x'}, \rho_{z'} \in  S^{n-1} $ so that
$$\inf_{\nu_{x'},\nu_{z'} \in  S^{n-1} } Tf ( x', z',P_{\nu_{x'}} ,P_{\nu_{z'}}  ) \ge  Tf (  x', z',P_{\rho_{x'}} , P_{\rho_{z'}}  ) - 2\eta . $$
Then we calculate that
\begin{align*} &  \mathbf{[II]} \le  \frac{1}{2} \bigg[ \alpha \bigg\{ u_{\epsilon} \bigg( x+\epsilon \rho_{x'}, t'- \frac{\epsilon^{2}}{2} \bigg) - u_{\epsilon} \bigg( z'+ \epsilon \rho_{z'}, t'- \frac{\epsilon^{2}}{2} \bigg) \bigg\} \\
& + \beta \kint_{B_{\epsilon}^{e_{1}} } \bigg\{ u_{\epsilon}\bigg(x'+ P_{\rho_{x'}} h,t'-\frac{\epsilon^{2}}{2} \bigg) - u_{\epsilon}\bigg(z'+ P_{\rho_{z'}} h,t'-\frac{\epsilon^{2}}{2} \bigg)  \bigg\} d \mathcal{L}^{n-1}(h)   \bigg]
\\ & \le \frac{1}{2} \bigg[ K +  \alpha f(x'+ \epsilon\rho_{x'},z'+ \epsilon \rho_{z'}) \\ & \qquad  + \beta \kint_{B_{\epsilon}^{e_{1}} }f(x'+ P_{\rho_{x'}}h,z'+P_{\rho_{z'}} h) d \mathcal{L}^{n-1}(h)+g\bigg( t'-\frac{\epsilon^{2}}{2}, s'-\frac{\epsilon^{2}}{2} \bigg)\bigg]
 \\& \le \frac{1}{2} \bigg[ K + Tf(x', z',P_{\rho_{x'}} , P_{\rho_{z'}}  ) +g\bigg( t'-\frac{\epsilon^{2}}{2}, s'-\frac{\epsilon^{2}}{2} \bigg) \bigg] \\ & \le
  \frac{1}{2} \bigg[ K +\inf_{\nu_{x'},\nu_{z'} \in  S^{n-1} }  Tf (  x', z',P_{\nu_{x'}} , P_{\nu_{z'}} ) + g\bigg( t'-\frac{\epsilon^{2}}{2}, s'-\frac{\epsilon^{2}}{2} \bigg)  \bigg] +\eta 
 . 
\end{align*}
We used (\ref{ineq}) again in the second inequality.

Combining the estimate for $  \mathbf{[I]}$ and  $  \mathbf{[II]}$, we obtain
\begin{align*} K  \le & \  u_{\epsilon}(x',t') - u_{\epsilon}(z',s') - H(x',z',t',s') + \eta \\  \le & \ K +\midrg_{\nu_{x'},\nu_{z'} \in  S^{n-1} }  Tf (   x', z,P_{\nu_{x'}} , P_{\nu_{z'}} ) +g\bigg( t'-\frac{\epsilon^{2}}{2}, s'-\frac{\epsilon^{2}}{2} \bigg)  \\ &
\ \ \ -H(x', z', t', s') + 2\eta . \end{align*}
Since $\eta$ is arbitrarily chosen, if we show that
\begin{align*}\midrg_{\nu_{x'},\nu_{z'} \in  S^{n-1} } Tf (  x', z',P_{\nu_{x'}} , P_{\nu_{z'}}) + g\bigg( t'-\frac{\epsilon^{2}}{2}, s'-\frac{\epsilon^{2}}{2} \bigg)   < H(x', z', t', s') ,\end{align*}
that is, \begin{align} \begin{split} \label{eq:hfirst} \midrg_{\nu_{x'},\nu_{z'} \in  S^{n-1} } Tf (  x', z',& P_{\nu_{x'}} , P_{\nu_{z'}} )  - f(x', z') \\ & < g(t',s') - g\bigg( t'-\frac{\epsilon^{2}}{2}, s'-\frac{\epsilon^{2}}{2} \bigg) ,   \end{split}
\end{align}
then the proof is completed. 

Now we need to estimate \eqref{eq:hfirst}. Without loss of generality, we assume that $t' \ge s'$. Then we see that
\begin{align*}g(t',s') - & g\bigg( t'-\frac{\epsilon^{2}}{2}, s'-\frac{\epsilon^{2}}{2} \bigg) \\ &= M ( |s'- r^{2}|^ {{\delta}/2} - r^{\delta}) - M \bigg( \bigg|s'-\frac{\epsilon^{2}}{2}- r^{2}\bigg|^ {{\delta}/2} - r^{\delta} \bigg)   
\\ & = M  |s'- r^{2}|^ {{\delta}/2} -M \bigg|s'-\frac{\epsilon^{2}}{2}- r^{2}\bigg|^ {{\delta}/2} .
\end{align*}
Note that 
$$  M  |s'- r^{2}|^ {{\delta}/2} -M \bigg|s'-\frac{\epsilon^{2}}{2}- r^{2}\bigg|^ {{\delta}/2} \ge M \bigg\{ r^{\delta} - \bigg( r^{2} + \frac{\epsilon^{2}}{2} \bigg)^{\frac{\delta}{2}} \bigg\} $$ 
and
$$ \bigg( r^{2} + \frac{\epsilon^{2}}{2} \bigg)^{\frac{\delta}{2}} \le  r^{\delta} +  \bigg( \frac{\epsilon^{2}}{2} \bigg)^{\frac{\delta}{2}} \le r^{\delta} + \epsilon^{\delta}$$ for $0 <\delta \le 1$. We also deduce that
\begin{align*}M  |s'- r^{2}|^ {{\delta}/2} -M \bigg|s'-\frac{\epsilon^{2}}{2}- r^{2}\bigg|^ {{\delta}/2} \ge -M \frac{\delta}{2} | s' - r^{2}|^{\frac{\delta}{2}-1} \frac{\epsilon^{2}}{2}
\ge -M r^{\frac{\delta}{2}-1} \epsilon^{2} ,
\end{align*}
 since $h(t)=|t|^{\delta / 2}$ is concave.

Therefore, we see that $$  g(t',s') - g\big( t'-\frac{\epsilon^{2}}{2}, s'-\frac{\epsilon^{2}}{2} \big)   \ge \min \{  -M \epsilon^{\delta} , -M \tilde{C}(r) \epsilon^{2} \} =: \sigma .$$

To establish (\ref{eq:hfirst}), we will distinguish several cases. And from now on, we will write $ (x, z, t, s) $ instead of $ (x', z', t', s')$ in our calculations for convenience.

\subsection{Case \texorpdfstring{$|x-z| > N \epsilon / 10$}.}

In this case, $f(x,z)=f_{1}(x,z)$ as $ f_{2}(x,z)=0 $. 
Thus we can write (\ref{eq:hfirst}) as 
 \begin{align} \label{eq:hsecond} \midrg_{\nu_{x},\nu_{z} \in  S^{n-1} }  T f_{1}(  x, z, P_{\nu_{x}} , P_{\nu_{z}} ) - f_{1}(x, z)  < \sigma   .  \end{align}
For any $\eta > 0$, we can choose some vectors $\nu_{x}, \nu_{z} \in S^{n-1}  $ and related rotations $P_{\nu_{x}} \in \mathbf{R}_{\nu_{x}}$, $P_{\nu_{z}} \in \mathbf{R}_{\nu_{z}}$ so that \begin{align*}\sup_{h_{x},h_{z} \in  S^{n-1} }  T f_{1} (   x, z, P_{h_{x}}, P_{h_{z}} ) \le  T f_{1} (   x, z,P_{\nu_{x}} , P_{\nu_{z}} )+\eta.\end{align*} Hence if we find some unit vectors $ \mu_{x}, \mu_{z} $ and rotations $P_{\mu_{x}} , P_{\mu_{z}} $ such that  
\begin{align*} \label{eq:ff} \midrg_{h_{x},h_{z} \in  S^{n-1} }& T f_{1} (   x, z, P_{h_{x}}, P_{h_{z}}  )\\ & \le \frac{1}{2} \big\{ T f_{1} (   x, z, P_{\nu_{x}} , P_{\nu_{z}} ) 
+ T f_{1} (  x, z, P_{\mu_{x}} , P_{\mu_{z}}) + \eta \big\} ,  \end{align*} then we obtain (\ref{eq:hsecond}) by showing
\begin{align} \begin{split}  \frac{1}{2} \big\{ T f_{1} (   x, z, P_{\nu_{x}} , P_{\nu_{z}} )  
+ T f_{1} (  x, z, P_{\mu_{x}} , P_{\mu_{z}}) \big\} - f_{1}(x,z) <  \sigma  - \eta .  \end{split}
\end{align}
Denote $\mathbf{v}=\frac{x-z}{|x-z|}$, $ y_{V} = \big\langle y, \mathbf{v} \big\rangle $ and $ y_{V_{\perp}}= y- y_{V}\mathbf{v}$. Then $y$ is orthogonally decomposed into $ y_{V}\mathbf{v}$ and $ y_{V_{\perp}}$. By using Taylor expansion, we know that for any $h_{x} $ and $h_{z} $,
\begin{align*} f_{1}&(x+\epsilon h_{x}, z+\epsilon h_{z}  ) \\ & =  f_{1}(x, z) + C \delta |x-z|^{\delta -1}(h_{x}-h_{z})_{V} \epsilon + 2M\langle x+z , h_{x}+h_{z} \rangle  \epsilon \\ & + 
\frac{1}{2}C \delta  |x-z|^{\delta -2} \big\{ (\delta -1)(h_{x}-h_{z})_{V}^{2} + |(h_{x}-h_{z})_{V^{\perp}}|^{2} \big\} \epsilon^{2} \\ & + M |  h_{x}+h_{z}|^{2} \epsilon^{2} + \mathcal{E}_{x,z}( \epsilon h_{x}, \epsilon h_{z}) ,
 \end{align*}
where $ \mathcal{E}_{x,z}(h_{x},h_{z})$ is the second-order error term.
Now we estimate the error term by Taylor's theorem as follows:
$$ | \mathcal{E}_{x,z}( \epsilon h_{x}, \epsilon h_{z}) | \le C |( \epsilon h_{x} , \epsilon  h_{z})^{t}  |^{3} (|x-z|-2\epsilon )^{\delta -3} $$
if $|x-z| > 2 \epsilon$. Thus if we choose $ N \ge \frac{100C}{\delta} $, we get
$$  | \mathcal{E}_{x,z}( \epsilon h_{x}, \epsilon h_{z}) | \le 10  |x-z|^{\delta -2}  \epsilon^{2}. $$

Now we establish \eqref{eq:hsecond}. 
We first consider a small constant $0<\Theta<4$ to be determined later
and we divide again this case into two separate subcases. 
In the first subsection, we consider the case when $\nu_{x}$, $\nu_{z}$ are in almost opposite directions and nearly parallel to the vector $ x- z $. 
Otherwise, it is covered in the second subsection.
In each case, we will choose proper rotations and investigate changes in the value of the auxiliary function $f_{1}$.
The concavity of $f_{1}$ plays a key role in both cases.

\subsubsection{Case \texorpdfstring{$(\nu_{x}-\nu_{z})_{V}^{2} \ge (4 - \Theta)  $}.}
Observe that
\begin{align*}  \midrg_{\nu_{x},\nu_{z} \in S^{n-1}} & T  f_{1}(  x, z, P_{\nu_{x}} , P_{\nu_{z}}  )\\ & \le  \frac{1}{2} \big\{ T f_{1}(  x, z, P_{\nu_{x}} , P_{\nu_{z}}  ) 
+ T f_{1}( x, z, -P_{\nu_{x}} , -P_{\nu_{z}}  ) + \eta \big\}  \end{align*} and 
\begin{align*}  \frac{1}{2} & \big\{ T  f_{1}(  x, z,P_{\nu_{x}} , P_{\nu_{z}} ) + T f_{1}(  x, z, -P_{\nu_{x}} , -P_{\nu_{z}}  ) \big\} -  f_{1}(x, z) \\ & =
\frac{\alpha}{2} \big\{  f_{1}(x+\epsilon \nu_{x}, z+\epsilon \nu_{z}) + f_{1}(x-\epsilon \nu_{x}, z-\epsilon \nu_{z})   - 2  f_{1}(x, z) \big\} \\  & \  +
\frac{\beta}{2} \bigg\{  \kint_{B_{\epsilon}^{e_{1}} }f_{1}(x+P_{\nu_{x}} h,z+P_{\nu_{z}} h) d \mathcal{L}^{n-1}(h) \\ & \qquad +  \kint_{B_{\epsilon}^{e_{1}} }f_{1}(x- P_{\nu_{x}}h,z-P_{\nu_{z}} h) d \mathcal{L}^{n-1}(h)   - 2  f_{1}(x, z) \bigg\}.
\end{align*}
We first estimate the $\alpha$-term. Using the Taylor expansion of $f_{1}$ and the above estimates, we get
\begin{align*}  &  f_{1}(x+\epsilon\nu_{x}, z+\epsilon\nu_{z})  + f_{1}(x-\epsilon \nu_{x}, z-\epsilon \nu_{z})   - 2  f_{1}(x, z) \\
& =  C \delta |x-z|^{\delta -2} \big\{ (\delta -1)(\nu_{x}-\nu_{z})_{V}^{2} + |(\nu_{x}-\nu_{z})_{V^{\perp}}|^{2} \big\}\epsilon^{2}
+2 M |  \nu_{x}+\nu_{z}|^{2}\epsilon^{2} \\ &   \ \  \  +  \mathcal{E}_{x, z}(\epsilon  \nu_{x},\epsilon \nu_{z}) +  \mathcal{E}_{x,z}(-\epsilon \nu_{x},-\epsilon \nu_{z})
 \\&
\le    C \delta |x-z|^{\delta -2} \{ (\delta -1) (4- \Theta)  + \Theta  \}\epsilon^{2} +2 M (2\epsilon)^{2} + 20|x-z|^{\delta -2} \epsilon^{2}  \\&
\le \big[  C \delta |x-z|^{\delta -2} \{ (\delta -1) (4- \Theta) + \Theta \} + 8M+ 20 |x-z|^{\delta -2} \big]  \epsilon^{2}.
\end{align*}
And note that \begin{align} \label{eq:hfourth} |P_{\nu_{x}}h -P_{\nu_{z}}h | \le |\nu_{x}+\nu_{z}|, \end{align}
for some proper $P_{\nu_{x}}, P_{\nu_{z}} $ and for any $ h \in B_{1}^{e_{1}} $(see \cite[ Appendix A]{alpr2018lipschitz}), to see that
\begin{align*} & \kint_{B_{\epsilon}^{e_{1}} }  f_{1}(x+ P_{\nu_{x}}h,z+P_{\nu_{z}} h) d \mathcal{L}^{n-1}(h)   -f_{1}(x, z) \\& \
=  \kint_{B_{\epsilon}^{e_{1}} } \bigg[ C \delta |x-z|^{\delta -1}(P_{\nu_{x}} h-P_{\nu_{z}} h )_{V}+2M\langle x+z ,P_{\nu_{x}} h+P_{\nu_{z}} h \rangle  \\& 
\qquad \qquad +\frac{C}{2} |x-z|^{\delta -2}\big\{ (\delta -1)(P_{\nu_{x}}h-P_{\nu_{z}} h)_{V}^{2} + |(P_{\nu_{x}} h-P_{\nu_{z}} h)_{V^{\perp}}|^{2} \big\} \\& 
\qquad \qquad +M| P_{\nu_{x}}h+P_{\nu_{z}} h|^{2} + \mathcal{E}_{x,z}(h_{x},h_{z}) \bigg]  d \mathcal{L}^{n-1}(h) \\& \
= \frac{1}{2} \kint_{B_{\epsilon}^{e_{1}} } \bigg[ C |x-z|^{\delta -2} \big\{ (\delta -1)(P_{\nu_{x}}h-P_{\nu_{z}} h)_{V}^{2} + |(P_{\nu_{x}}h-P_{\nu_{z}} h)_{V^{\perp}}|^{2} \big\} \\&
\qquad \qquad + 2M| P_{\nu_{x}}h+P_{\nu_{z}} h|^{2} + 2 \mathcal{E}_{x,z}(h_{x},h_{z}) \bigg]  d \mathcal{L}^{n-1}(h) \\& \
\le  \frac{1}{2}  \big\{  |x-z|^{\delta -2} (C \Theta + 20) + 8M \big\} \epsilon ^{2}.
\end{align*} The last inequality follows from $ |\nu_{x}+\nu_{z}| ^{2} \le \Theta $. In the same way, it is also obtained
\begin{align*}\kint_{B_{\epsilon}^{e_{1}} } f_{1}(x-P_{\nu_{x}}h,&z-P_{\nu_{z}} h) d \mathcal{L}^{n-1}(h)   -f_{1}(x, z) \\ & \le \frac{1}{2}  \big\{  |x-z|^{\delta -2} (C \Theta + 20) + 8M \big\} \epsilon ^{2} .
\end{align*}
These estimates give \begin{align*}  \frac{1}{2} & \big\{ T f_{1}(  x, z, \nu_{x}, \nu_{z} ) + T f_{1}(  x, z, -\nu_{x}, -\nu_{z} ) \big\} -  f_{1}(x, z) \\ & \le
\frac{\alpha}{2} \big[  C \delta |x-z|^{\delta -2} \{ (\delta -1) (4- \Theta) + \Theta \} + 8M+ 20 |x-z|^{\delta -2} \big]  \epsilon^{2} \\  & \  +
\frac{\beta}{2}  \big\{ C  \Theta |x-z|^{\delta -2}  + 8M + 20|x-z|^{\delta -2} \big\} \epsilon ^{2} \\ & \le
\bigg[ \frac{C}{2} \big\{ \Theta + \alpha \delta (\delta - 1)(4 - \Theta)   \big\} +10 \bigg] |x-z|^{\delta -2}   \epsilon ^{2}  + 4M \epsilon ^{2}.
\end{align*}
Observe that $ \Theta + \alpha \delta (\delta - 1)(4 - \Theta)  < 0 $ if $\Theta < 4\alpha \delta (1-\delta)/\{ 1- \alpha \delta (\delta -1) \}$.
Then we can choose sufficiently large $C$ depending only on $ r, \delta, \alpha$ and $n$  so that 
\begin{align*}  \midrg_{\nu_{x},\nu_{z} \in S^{n-1}}  Tf_{1}( x, z, P_{\nu_{x}}, P_{\nu_{z}} ) - f_{1}(x, z)  < -M \tilde{C} \epsilon^{2}. \end{align*}
Thus, we get (\ref{eq:hsecond}).

\subsubsection{Case $(\nu_{x}-\nu_{z})_{V}^{2} \le (4 - \Theta) $}
It  is clear that $ |\nu_{x}-\nu_{z}|_{V} < 2- \Theta/4 $ in this case. Furthermore, we check that
\begin{align}  \label{show2} 
\begin{split}  \midrg_{\nu_{x},\nu_{z} \in S^{n-1}} & Tf_{1}(  x, z,  P_{h_{x}}, P_{h_{z}} ) \\ & \le \frac{1}{2}  \big\{ T f_{1}(  x, z, P_{\nu_{x}}, P_{\nu_{z}} ) + T f_{1} ( x, z, P_{-\mathbf{v}}, P_{\mathbf{v}}) \big\} + \eta .  \end{split}
\end{align}
Now we estimate the right hand side. By the DPP, it can be written as
\begin{align*} \frac{1}{2} &  \big\{  T f_{1}(  x, z, P_{\nu_{x}}, P_{\nu_{z}} ) + T f_{1} ( x, z, P_{-\mathbf{v}}, P_{\mathbf{v}})\big\} - f_{1}(x,z) \\ &
 = \frac{\alpha}{2} \big\{  f_{1}(x+ \epsilon \nu_{x}, z+\epsilon \nu_{z}) + f_{1}(x-\epsilon \mathbf{v}, z+\epsilon \mathbf{v})   - 2  f_{1}(x, z) \big\} \\  & \  +
\frac{\beta}{2} \bigg\{  \kint_{B_{\epsilon}^{e_{1}} }f_{1}(x+P_{\nu_{x}} h,z+P_{\nu_{z}} h) d \mathcal{L}^{n-1}(h)  \\ & \qquad \qquad +  \kint_{B_{\epsilon}^{e^{1}} }f_{1}(x+  P_{-\mathbf{v}}h,z+ P_{\mathbf{v}} h) d \mathcal{L}^{n-1} (h)  - 2  f_{1}(x, z) \bigg\} .
\end{align*}
We will continue in a similar way to the previous case. For the $\alpha$-term, we deduce that
\begin{align*}  &f_{1}(x+ \epsilon \nu_{x}, z+\epsilon \nu_{z}) + f_{1}(x-\epsilon \mathbf{v}, z+\epsilon \mathbf{v})   - 2  f_{1}(x, z)  \\ &
=\frac{1}{2} \bigg[  C \delta |x-z|^{\delta -1} \{ (\nu_{x}-\nu_{z})_{V} - 2 \}\epsilon + 2M\langle x+z , \nu_{x}+\nu_{z} \rangle \epsilon \\ &
\ \ \ + \frac{C}{2} \delta  |x-z|^{\delta -2} \{  (\delta -1)( (\nu_{x}-\nu_{z})_{V}^{2}\epsilon^{2} + (2 \epsilon )^{2} ) + |(\nu_{x}-\nu_{z})_{V^{\perp}}|^{2}\epsilon^{2}  \} \\ &
\ \ \ + 4M \epsilon^{2}+ M |  \nu_{x}+\nu_{z}|^{2}\epsilon^{2}  +  \mathcal{E}_{x,z}(\epsilon \nu_{x},\epsilon \nu_{z}) +  \mathcal{E}_{x,z}(-\epsilon \mathbf{v},\epsilon \mathbf{v} ) \bigg] \\ &
\hspace{-0.25em}  \le \frac{1}{2} \bigg\{ \hspace{-0.25em} - \hspace{-0.25em} \frac{\Theta}{4} C \delta |x-z|^{\delta -1} \epsilon \hspace{-0.15em} + \hspace{-0.15em} 8M\epsilon r \hspace{-0.15em}+\hspace{-0.15em} 2C \delta  |x-z|^{\delta -2}\epsilon^{2}\hspace{-0.15em} +\hspace{-0.15em} 20|x-z|^{\delta -2}\epsilon^{2}\hspace{-0.15em} +\hspace{-0.15em} 2M\epsilon^{2} \bigg\} .
\end{align*}
Then we see that 
\begin{align*} 2C & \delta  |x-z|^{\delta -2}\epsilon^{2} + 20|x-z|^{\delta -2}\epsilon^{2}+2M\epsilon^{2}
\\ & \le \frac{10}{N} (2C \delta + 20 + 2M \diam(\Omega)^{2-\delta }) |x-z|^{\delta -1}\epsilon
 \\ & \le \delta^{2} |x-z|^{\delta -1}\epsilon
\end{align*}
for sufficiently large $C$ and $N \ge 100C / \delta$ , since $|x-z| > N \epsilon / 10$ and $\Omega$ is bounded. Thus,
\begin{align*} f_{1}(x+\nu_{x}, z+\nu_{z})& + f_{1}(x-\epsilon \mathbf{v}, z+\epsilon \mathbf{v})   - 2  f_{1}(x, z) \\ &
\le \bigg\{\frac{\delta}{2} |x-z|^{\delta -1}\bigg( \delta - C\frac{\Theta}{4} \bigg)   + 4Mr \bigg\}\epsilon .
\end{align*}
Next, we estimate the $\beta$-term. By a direct calculation, we see that
\begin{align*}
& \hspace{-0.2em}  \kint_{B_{\epsilon}^{e_{1}} } \hspace{-0.3em}  \big\{  f_{1}(x+P_{\nu_{x}} h,z+P_{\nu_{z}} h)+f_{1}(x+ P_{-\mathbf{v}}h,z+P_{\mathbf{v}} h)  - 2  f_{1}(x, z) \big\} d \mathcal{L}^{n-1}(h) \\ &
=\kint_{B_{\epsilon}^{e_{1}}} \{ f_{1}(x+ P_{\nu_{x}}h,z+P_{\nu_{z}} h) - f_{1}(x, z) \} d \mathcal{L}^{n-1}(h) \\ & \qquad \qquad +  \kint_{B_{\epsilon}^{e_{1}} } \{ f_{1}(x+ P_{-\mathbf{v}}h,z+P_{\mathbf{v}}h) - f_{1}(x, z) \} d \mathcal{L}^{n-1}(h) \\ &
\le \kint_{B_{\epsilon}^{e_{1}}} \bigg[  \frac{C}{2} \delta |x-z|^{\delta-2} \big\{(\delta -1)(P_{\nu_{x}}h-P_{\nu_{z}} h)_{V}^{2} + |(P_{\nu_{x}}h-P_{\nu_{z}} h)_{V^{\perp}}|^{2} \big\}  \\&
\qquad \qquad +M|P_{\nu_{x}} h+P_{\nu_{z}} h|^{2} + \mathcal{E}_{x,z}(h_{x},h_{z}) \bigg]  d \mathcal{L}^{n-1}(h) \\&
\ \ \ +  \kint_{B_{\epsilon}^{e_{1}} } \bigg[  \frac{C}{2}  \delta |x-z|^{\delta-2}(2h)^{2} + \mathcal{E}_{x,z}(-\epsilon \mathbf{v},\epsilon\mathbf{v} ) \bigg] d \mathcal{L}^{n-1}(h) \\&
\le  C\delta |x-z|^{\delta-2}(2 \epsilon)^{2} + M(2 \epsilon)^{2} +  20|x-z|^{\delta -2} \epsilon^{2} ,
\end{align*}
we have used (\ref{eq:hfourth}) for the last inequality.
Now we observe that
$$ C\delta |x-z|^{\delta-2}(2 \epsilon)^{2} + M(2 \epsilon)^{2} +  20|x-z|^{\delta -2} \epsilon^{2} \le  2 \delta^{2} |x-z|^{\delta -1}\epsilon. $$ Therefore $\beta$-term is estimated by $ 2 \delta^{2} |x-z|^{\delta -1}\epsilon. $

Combining these estimates, we conclude
\begin{align*}  \frac{1}{2} &  \big\{ Tf_{1}(  x, z, P_{\nu_{x}}, P_{\nu_{z}} ) + T f_{1} (  x, z, P_{-\mathbf{v}},  P_{\mathbf{v}}) \big\} - f_{1}(x,z)  \\ &
\le \frac{\alpha}{2}  \bigg\{\frac{\delta}{2} |x-z|^{\delta -1}\bigg( \delta - C\frac{\Theta}{4} \bigg)   + 4Mr \bigg\}\epsilon + \beta  \delta^{2} |x-z|^{\delta -1}\epsilon \\&
\le -M \tilde{C} \epsilon^{2} 
\end{align*}
for sufficiently large $C=C(r,\delta, \alpha, n)$. Combining this with \eqref{show2}, we obtain (\ref{eq:hfirst}). 
\subsection{Case \texorpdfstring{$0 < |x-z| \le N \epsilon / 10$}.} We observe that
\begin{align*} | f_{1}&(x+h_{x}, z+h_{z}) - f_{1}(x,z) |  \\ &
=C ( |x -z +h_{x} - h_{z}|^{\delta} - |x-z|^{\delta} )+ M( |x +z +h_{x} + h_{z}|^{2} - |x+z|^{2}) \\&
 \le C|h_{x} - h_{z}|^{\delta} + 2M|x+z| \ |h_{x} +h_{z}|+M|h_{x} +h_{z}|^{2} \\&
\le 2C \epsilon^{\delta} + 8Mr\epsilon + 4M \epsilon^{2} \\&
\le 3C\epsilon^{\delta}
\end{align*}
 for any $x , z \in B_{r}$ and $h_{x}, h_{z} \in B_{\epsilon}$ if $C=C(r, \delta)$ is sufficiently large. Therefore, we see that
\begin{align*} & \sup_{h_{x},h_{z} \in S^{n-1}} T f_{1}( x, z, P_{h_{x}}, P_{h_{z}}  ) - f_{1}(x,z)  \\& \  = \sup_{h_{x},h_{z} \in S^{n-1}} \bigg[ \alpha\{ f_{1}(x+ h_{x},z+ h_{z})-f_{1}(x,z)   \}  + \\ & \qquad \qquad \beta \kint_{B_{\epsilon}^{e_{1}} }\{f_{1}(x+ P_{h_{x}}h,z+P_{h_{z}} h) - f_{1}(x,z) \}  d \mathcal{L}^{n-1}(h) \bigg] \\& \
\le 3\alpha C\epsilon^{\delta}+ 3\beta C\epsilon^{\delta} = 3C\epsilon^{\delta}
\end{align*} and
\begin{align} \begin{split} \label{anes} \sup_{h_{x},h_{z} \in S^{n-1}} T f(  x, z,P_{h_{x}}, P_{h_{z}}  )& =\sup_{h_{x},h_{z} \in S^{n-1}} T( f_{1}-f_{2})(x, z, P_{h_{x}}, P_{h_{z}}  ) \\ & \le \sup_{h_{x},h_{z}  \in S^{n-1}} Tf_{1}(  x, z, P_{h_{x}}, P_{h_{z}}  ) . \end{split} \end{align}

By the assumption, we can find $i \in \{1, 2, \cdots , N\}$ such that  $$ (i-1) \frac{ \epsilon}{10} < |x-z| \le i\frac{ \epsilon}{10}. $$
We deduce that
\begin{align*} & \inf_{h_{x},h_{z} \in S^{n-1}} T f(  x, z, P_{h_{x}}, P_{h_{z}}  ) \\ & \le \sup_{h_{x},h_{z} \in S^{n-1}} T f_1(  x, z, P_{h_{x}}, P_{h_{z}}  )  - \sup_{h_{x},h_{z} \in S^{n-1}} T f_{2}(  x, z, P_{h_{x}}, P_{h_{z}}  ) \\&
\le  \sup_{h_{x},h_{z} \in S^{n-1}} Tf_1(  x, z,P_{h_{x}}, P_{h_{z}}  ) - \alpha C^{2(N-i+1)} \epsilon^{\delta} \\&
= \sup_{h_{x},h_{z} \in S^{n-1}} T f_1(  x, z, P_{h_{x}}, P_{h_{z}}  ) - \alpha \bigg( C^{2}- \frac{2}{\alpha} \bigg) C^{2(N-i)} \epsilon^{\delta} - 2C^{2(N-i)} \epsilon^{\delta} \\&
\le \sup_{h_{x},h_{z} \in S^{n-1}} T f_1(  x, z,P_{h_{x}}, P_{h_{z}}  ) - 2f_{2}(x,z) - 8C \epsilon ^{\delta}.
\end{align*}
The last inequality is obtained if $C$ is large. Therefore, we calculate that
\begin{align*}  \midrg_{\nu_{x},\nu_{z} \in S^{n-1}} T f(  x, z, P_{h_{x}}, P_{h_{z}} ) & \le   \hspace{-0.45em} \sup_{h_{x},h_{z} \in S^{n-1}} \hspace{-0.25em} T f_1(  x, z,P_{h_{x}}, P_{h_{z}}  ) - f_{2}(x,z)-4C\epsilon^{\delta}
\\& \le  f_{1}(x,z) + 3C\epsilon^{\delta} -  f_{2}(x,z)-4C\epsilon^{\delta}
\\& \le f(x,z) - C\epsilon^{\delta},
\end{align*}
and then we get \eqref{eq:hfirst} for choosing $C= C(r, \delta, \alpha, n)$ sufficiently large. 
\subsection{Case \texorpdfstring{$|x-z| = 0$}.} According to the results in the previous sections, we observe that 
$$ |u_{\epsilon} (x,t) - u_{\epsilon} (z,s) | \le C_{1} ||u_{\epsilon}||_{\infty} ( |x-z|^{\delta} + \epsilon^{\delta}) ,$$ for any $x, z \ (x \neq z) \in B_{r}(0)$, $ -r^{2}<t<0$, $|t-s| < \epsilon^{2} / 2 $ and some $C_{1}= C_{1}(r, \delta, \alpha , n)>0$.

Fix $x \in B_{r}(0)$ and $t,s \in (-r^{2},0)$ with $|t-s| < \epsilon^{2} / 2$. 
Then we can choose a point $y \in B_{\epsilon}(x)$ and deduce that
\begin{align*}
|u_{\epsilon} (x,t) - u_{\epsilon} (x,s) | & \le |u_{\epsilon} (x,t) - u_{\epsilon} (y,s) | +|u_{\epsilon} (y,s) - u_{\epsilon} (x,s) | 
\\ & \le C_{1} ||u_{\epsilon}||_{\infty} ( |x-y|^{\delta} + \epsilon^{\delta})
\\ & \le 2C_{1} ||u_{\epsilon}||_{\infty} \epsilon^{\delta}.
\end{align*}
Now set $C =2C_{1}$. Then we can conclude the proof of this lemma.
\end{proof}

For any $x \in B_{r}$ and $ -r^{2} < s < t < 0  $, consider a cylinder $B_{\sqrt{t-s}}(x) \times [s,t] $. Applying Lemma \ref{lem1}, we find that
\begin{align*} \osc_{B_{\sqrt{t-s}}(x) \times \big(\tau-\frac{ \epsilon^{2}}{2},\tau \big) } u_{\epsilon} 
  \le C(r, \delta, \alpha, n) ||u_{\epsilon}||_{\infty} (|t-s|^{\frac{\delta}{2}} + \epsilon^{\delta})
\end{align*} for any $ \tau \in (s,t)$.
Then we obtain the following estimate 
\begin{align*} |u_{\epsilon} (x,t) - u_{\epsilon} (x,s) |   \le C(r, \delta, \alpha, n) ||u_{\epsilon}||_{\infty} (|t-s|^{\frac{\delta}{2}} + \epsilon^{\delta})
\end{align*} by virtue of Lemma \ref{lem2}.

Combining this and Lemma \ref{lem1}, we get the desired regularity.

\begin{theorem} \label{thm1} 
Let $ \bar{Q}_{2r}  \subset \Omega_{T}$, $ 0 <  \delta, \alpha  < 1 $ and $ \epsilon > 0 $ is small.
Suppose that $u_{\epsilon}$ satisfies the $\alpha$-parabolic DPP with boundary data $F \in L^{\infty}(\Gamma_{\epsilon, T} )$.
Then for any  $x, z \in B_{r}(0)$ and $ -r^{2} < t, s< 0  $,
$$ |u_{\epsilon} (x,t) - u_{\epsilon} (z,s) | \le C ||u_{\epsilon}||_{\infty} ( |x-z|^{\delta}+ |s-t|^{\frac{\delta}{2}} + \epsilon^{\delta}), $$ 
where $C>0$ is a constant which only depends on $r, \delta, \alpha$ and $n$.
\end{theorem}

\section{Lipschitz regularity}
We will prove Lipschitz type regularity for the function $u_{\epsilon}$ in this section. 
In the previous section, we utilized the concavity on the distance of two points of the auxiliary function to get the result. 
In order to prove Lipschitz estimate, the auxiliary function is also needed to have this property.
However, we no longer have the strong concavity that was helpful in the proof there.
Therefore, we need to build the proof in a different manner in several places.

For this reason, we will construct other (concave) auxiliary function for proving Lipschitz estimate. This causes some difficulties compared to the H\"{o}lder case. 
As in the proof of Lemma \ref{lem1}, we will distinguish two subcases. More delicate calculations are needed when two points are sufficiently far apart.
Note that we will exploit the H\"{o}lder regularity result here.
In the case that two points are sufficiently close, the proof is quite similar to the previous section.
 
\begin{lemma} \label{lem3}
Let $\bar{B}_{2r}(0) \times [-2r^{2}-\epsilon^{2} / 2, \epsilon^{2}/2] \subset \Omega_{T}$, $ 0< \alpha <1 $ and $\epsilon > 0 $ is small. 
Suppose that $u_{\epsilon}$ satisfies the $\alpha$-parabolic DPP with boundary data $F \in L^{\infty}(\Gamma_{\epsilon, T} )$.
Then,
$$ |u_{\epsilon} (x,t) - u_{\epsilon} (z,s) | \le C||u_{\epsilon}||_{\infty} ( |x-z| + \epsilon) ,$$ whenever $x, z \in B_{r}(0)$, $ -r^{2}<t<0$ and $ |t-s| < \epsilon^{2} / 2 $ and $C>0$ is a constant which only depends on $r, \alpha$ and $n$.
\end{lemma}

\begin{proof} 
We can expect that $ |x-z| $ will play the same role as $f_{1}$ in the H\"{o}lder case. But for a Lipschitz type estimate, we cannot deduce the desired result by using that function $ |x-z| $. Therefore, we need to define a new auxiliary function $\omega : [0, \infty) \to [0, \infty)$. First define $$\omega(t) = t- \omega_{0}t^\gamma  \qquad 0 \le t \le \omega_{1} := (2\gamma \omega_{0})^{-1/(\gamma-1)},$$ where $\gamma \in (1,2)$ is a constant and $ \omega_{0} >0$ will be determined later. 
Observe that $$ \omega '(t) = 1-\gamma\omega_{0}t^{\gamma-1} \in [ 1/2,1 ] \qquad \textrm{for} \ \ 0 \le t \le \omega_{1}$$ and$$ \omega ''(t) = -\gamma(\gamma-1)\omega_{0}t^{\gamma-2} <0 \qquad \textrm{for} \ \ 0 \le t \le \omega_{1}.$$
Then we can construct $\omega$ to be increasing, strictly concave and $C^{2}$ in $(0, \infty)$.

Assume that $ ||u_{\epsilon}||_{\infty} \le r  $ by scaling as in the previous section, and we define $$f_{1}(x,z) = C \omega( | x - z | ) + M | x+ z|^{2}. $$
Consider the functions $ f_{2} $ and $g$ for $ \delta = 1$ as (\ref{f2}) and (\ref{gg}), respectively.
Now we set again the auxiliary function $H$ by $$ H(x, z, t, s) = f_{1}(x,z) - f_{2}(x,z) + g(t,s)$$
and let $$ f(x,z)= f_{1}(x,z) - f_{2}(x,z).$$
As in the previous section, we will first deduce that
$$ |u_{\epsilon}(x,t) - u_{\epsilon}(z,s) | \le  C( |x-z|+ \epsilon ) \qquad \textrm{in} \ \  \  \ \Sigma_{2} \backslash \Upsilon . $$
We can choose $M$ sufficiently large so that $$ u_{\epsilon}(x,t) - u_{\epsilon}(z,s) - H (x,z,t,s) \le  C ^ {2N} \epsilon +  C \epsilon \qquad  \textrm{in} \ \ \ \Sigma_{2} \backslash \Sigma_{1}  .  $$ 
Thus, for proving the lemma, it is sufficient to show that 
$$ u_{\epsilon}(x,t) - u_{\epsilon}(z,s) - H(x,z,t,s) \le  C ^ {2N} \epsilon +C \epsilon  \qquad  \textrm{in} \ \ \ \Sigma_{1}\backslash \Upsilon .  $$
Suppose not. Then \begin{align*} K := \sup_{(x,z,t,s) \in \Sigma_{1} \backslash \Upsilon} (u_{\epsilon}(x,t) - u_{\epsilon}(z,s) - H(x,z,t,s)) > C ^ {2N} \epsilon+C \epsilon   . \end{align*}
In this case, we can choose $(x', z', t', s') \in  \Sigma_{1} \backslash \Upsilon$ such that 
\begin{align}  \label{counter} u_{\epsilon}(x',t') - u_{\epsilon}(z',s') -H (x',z',t',s')) \ge K - \eta \end{align}
for any $ \eta > 0 $. 

Similarly as in Section 4, we need to establish (\ref{eq:hfirst}) in order to prove Lemma \ref{lem3}. The only difference is the right-hand side of the inequality.  
In this case, it is sufficient to deduce that the left-hand side of (\ref{eq:hfirst}) is less than $ \sigma = \min \{ -M \epsilon , -M \tilde{C} \epsilon^{2} \} $, where $\tilde{C}$ only depends on $r$.

We use again the notation $ (x, z, t, s) $ instead of $ (x', z', t', s') $.

\subsection{Case \texorpdfstring{$|x-z| > N \epsilon / 10$}.}
For the same reason as in the previous section, we shall deduce (\ref{eq:hsecond}). To do this, it is sufficient to show (\ref{eq:ff}) for any $\eta > 0$ and some $P_{\nu_{x}} \in \mathbf{R}_{\nu_{x}}$, $P_{\nu_{z}} \in \mathbf{R}_{\nu_{z}}$.

Now we calculate the Taylor expansion of $f_{1}$. We see \begin{align} \label{eq:third} \begin{split}
 f_{1}&(x+  \epsilon h_{x}, z+ \epsilon h_{z}  ) - f_{1}(x, z) \\ & 
\le  C \omega ' {(|x-z|)}(h_{x}-h_{z})_{V} \epsilon + 2M\langle x+z , h_{x}+h_{z} \rangle  \epsilon \\ & + \frac{1}{2}C \omega ''(|x-z|) (h_{x}-h_{z})_{V}^{2} \epsilon^{2} +
\frac{1}{2}C \frac{\omega'(|x-z|)}{|x-z|} |(h_{x}-h_{z})_{V^{\perp}}|^{2} \epsilon^{2}  \\ &
+ (4M +10  |x-z|^{\gamma -2})\epsilon^{2}\end{split}
\end{align}  for any $h_{x} ,h_{z} \in \mathbb{R}^{n}$.
Then we check that
$$ | \mathcal{E}_{x,z}( h_{x}, h_{z}) | \le C |(h_{x} , h_{z})^{t}  |^{3} (|x-z|-2\epsilon )^{\gamma -3} \le  C |(h_{x} , h_{z})^{t}  |^{3} (|x-z|-2\epsilon )^{\gamma -3}$$
if $|x-z| > 2 \epsilon$ and $h_{x}, h_{z} \in B_{\epsilon}$, because for the third derivatives it holds $D_{(x,z)}^{3} \omega (| x-z|) \le C |x-z|^{\gamma -3}$ for some constant $C>0$.  Thus if we choose $ N \ge \frac{100C}{\delta} $, we get
$$  | \mathcal{E}_{x,z}( h_{x},  h_{z}) | \le 10  |x-z|^{\gamma -2}  \epsilon^{2}. $$ 

For estimating $\alpha$-term in $T f_{1}(x, z, P_{\nu_{x}} , P_{\nu_{z}})$, we can use  (\ref{eq:third}) directly. On the other hand, more observations about $P_{\nu_{x}} , P_{\nu_{z}}$ are needed to estimate $\beta$-term. First we see that
\begin{align*} f_{1}&(x+  \epsilon P_{\nu_{x}} \zeta, z+ \epsilon P_{\nu_{z}} \zeta ) - f_{1}(x, z) \\ & = C \omega ' {(|x-z|)}( P_{\nu_{x}}\zeta-P_{\nu_{z}} \zeta )_{V} \epsilon + 2M\langle x+z , P_{\nu_{x}} \zeta +P_{\nu_{z}} \zeta  \rangle  \epsilon \\ & + \frac{1}{2}C \omega ''(|x-z|) (P_{\nu_{x}} \zeta -P_{\nu_{z}} \zeta )_{V}^{2} \epsilon^{2} +
\frac{1}{2}C \frac{\omega'(|x-z|)}{|x-z|} |(P_{\nu_{x}} \zeta -P_{\nu_{z}} \zeta )_{V^{\perp}}|^{2} \epsilon^{2} \\& +
  M |P_{\nu_{x}} \zeta+P_{\nu_{z}}\zeta|^{2}\epsilon^{2} + \mathcal{E}_{x,z}(\epsilon h_{x}, \epsilon h_{z}) 
\end{align*} from (\ref{eq:third}). Due to rotational symmetry, integral over the first-order terms is zero. Note that $ \omega '' < 0 $ and  (\ref{eq:hfourth}) to see that
\begin{align*} \kint_{B_{\epsilon}^{e_{1}} } & f_{1}(x+P_{\nu_{x}} h,z+P_{\nu_{z}} h) d \mathcal{L}^{n-1}(h) - f_{1}(x,z) \\ &
\le  \frac{C}{2} \frac{\omega'(|x-z|)}{|x-z|}|\nu_{x}+\nu_{z}|^{2}  \epsilon^{2} + (4M +10  |x-z|^{\gamma -2})\epsilon^{2}.
\end{align*} Therefore, 
\begin{align*} Tf_{1}  &(   x, z, P_{\nu_{x}}, P_{\nu_{z}} ) - f_{1}(x,z) \\ &
\le  \alpha C \omega ' {(|x-z|)}(\nu_{x}-\nu_{z})_{V} \epsilon + 2\alpha M\langle x+z , \nu_{x}+\nu_{z} \rangle  \epsilon \\ & + \frac{\alpha}{2}C \omega ''(|x-z|) (\nu_{x}-\nu_{z})_{V}^{2} \epsilon^{2} \\ & +
\frac{1}{2}C \frac{\omega'(|x-z|)}{|x-z|}( \alpha |(\nu_{x}-\nu_{z})_{V^{\perp}}|^{2} +\beta|\nu_{x}+\nu_{z}|^{2} ) \epsilon^{2} \\ &
+ (4M +10  |x-z|^{\gamma -2})\epsilon^{2}.
\end{align*}

Now we set $ \Theta = |x-z|^{s} $ for some $s \in (0,1]$ to be chosen later.
In order to deduce (\ref{eq:hsecond}), we divide again this case into two separate subcases.

\subsubsection{$(\nu_{x}-\nu_{z})_{V}^{2} \ge 4 - \Theta $}
Consider two rotations $P_{\nu_{x}}, P_{\nu_{z}} $ which satisfy (\ref{eq:hfourth}).  Observe that  
\begin{align} \label{eq:fifth} \begin{split}  \frac{1}{2} & \big\{ Tf_{1} ( x, z,P_{\nu_{x}}, P_{\nu_{z}}  ) + Tf_{1} (  x, z, -P_{\nu_{x}}, -P_{\nu_{z}} ) \big\} -  f_{1}(x, z) \\ & \le
\frac{\alpha}{2}C \omega ''(|x-z|) (\nu_{x}-\nu_{z})_{V}^{2} \epsilon^{2} \\ & +
\frac{1}{2}C \frac{\omega'(|x-z|)}{|x-z|}( \alpha |(\nu_{x}-\nu_{z})_{V^{\perp}}|^{2} + \beta|\nu_{x}+\nu_{z}|^{2} ) \epsilon^{2}  \\ & + (4M +10  |x-z|^{\gamma -2})\epsilon^{2}. \end{split}
\end{align} Since $ \Theta \le 1 $ for sufficiently small $r$ and $  \frac{1}{2} \le  \omega ' \le 1   $ and $ \omega '' < 0 $, 
\begin{align*}  \frac{1}{2} & \big\{ Tf_{1} ( x, z,P_{\nu_{x}}, P_{\nu_{z}}  ) + Tf_{1} (  x, z, -P_{\nu_{x}}, -P_{\nu_{z}} ) \big\} -  f_{1}(x, z) \\ & \le
\frac{3}{2}\alpha C \omega ''(|x-z|) \epsilon^{2}  +
\frac{C}{2} \frac{1}{|x-z|} ( \alpha |(\nu_{x}-\nu_{z})_{V^{\perp}}|^{2} +\beta|\nu_{x}+\nu_{z}|^{2} ) \epsilon^{2}  \\ & + (4M +10  |x-z|^{\gamma -2})\epsilon^{2}.
\end{align*} We know that $ |(\nu_{x}-\nu_{z})_{V^{\perp}}|^{2} \le \Theta $ by the assumption and we also see
$$  |\nu_{x}+\nu_{z}|^{2} = 4 -  |(\nu_{x} -\nu_{z} )|^{2} \le 4 -  |(\nu_{x} -\nu_{z} )_{V}|^{2} \le \Theta .$$ Thus,
\begin{align*}  \frac{1}{2} & \big\{ Tf_{1} ( x, z, P_{\nu_{x}}, P_{\nu_{z}}  ) + Tf_{1} (  x, z, -P_{\nu_{x}}, -P_{\nu_{z}}  ) \big\} -  f_{1}(x, z) \\ & \le
\bigg\{ \frac{3}{2}\alpha C \omega ''(|x-z|)  + \frac{C}{2} \frac{\Theta}{|x-z|} + 4M +10  |x-z|^{\gamma -2} \bigg\} \epsilon^{2} .
\end{align*}
By the definition of $\omega$, $\omega''(|x-z|)= -\gamma(\gamma-1)\omega_{0}|x-z|^{\gamma-2}$ 
if $|x-z| < \omega_{1} $.

Choosing $\gamma = 1+s$. Since $|x-z|<1$, we get 
\begin{align*}  \frac{1}{2} & \big\{ Tf_{1} ( x, z,P_{\nu_{x}}, P_{\nu_{z}}  ) + Tf_{1} (  x, z, -P_{\nu_{x}}, -P_{\nu_{z}}  ) \big\} -  f_{1}(x, z) \\ & \le
\bigg[  C \bigg\{- \frac{3}{2}\alpha s(s+1) \omega_{0}  + 11  \bigg\} |x-z|^{s-1}+ 4M  \bigg] \epsilon^{2} .
\end{align*} 
Note that if $|x-z|< \omega_{1}$ (See the definition of $\omega$),
$$- \frac{3}{2}\alpha s(s+1) \omega_{0}  + 11 < 0$$ for sufficiently large $\omega_{0}$.
Now we select $C= C(r, \alpha, n)$ sufficiently large so that 
\begin{align*} \bigg[  C \bigg\{- \frac{3}{2}\alpha s(s+1) \omega_{0}  + 11  \bigg\} |x-z|^{s-1}+ 4M  \bigg] \epsilon^{2} \le -M \tilde{C} \epsilon ^{2}
\end{align*}
then we get (\ref{eq:hfirst}).

\subsubsection{Case $(\nu_{x}-\nu_{z})_{V}^{2} < (4 - \Theta) $}
Consider two rotations  $P_{\mathbf{v}}$ and $P_{-\mathbf{v}}$ as follows: The first column vectors of $P_{-\mathbf{v}}$ and $P_{\mathbf{v}}$ are ${\mathbf{v}}$ and ${-\mathbf{v}}$, respectively. And other column vectors are the same. Then we observe,
\begin{align*} Tf_{1}  &(   x, z, P_{-\mathbf{v}},  P_{\mathbf{v}} ) - f_{1}(x,z) \\ &
\le -2 \alpha C \omega ' {(|x-z|)} \epsilon + 2\alpha C \omega ''(|x-z|)  \epsilon^{2}
+ (4M +10  |x-z|^{\gamma -2})\epsilon^{2} \\&
\le  -2 \alpha C \omega ' {(|x-z|)} \epsilon + (4M +10  |x-z|^{\gamma -2})\epsilon^{2},
\end{align*} and thus 
\begin{align*} \frac{1}{2} \big\{ Tf_{1} ( x, z,& P_{\nu_{x}}, P_{\nu_{z}} ) + Tf_{1} (  x, z,P_{-\mathbf{v}},  P_{\mathbf{v}}) \big\}  - f_{1}(x,z) \\ &
\le  \alpha C \omega ' {(|x-z|)}\{(\nu_{x}-\nu_{z})_{V} - 2 \} \epsilon + 2\alpha M\langle x+z , \nu_{x}+\nu_{z} \rangle  \epsilon \\ & + 
\frac{1}{2}C \frac{\omega'(|x-z|)}{|x-z|} \{ \alpha |(\nu_{x}-\nu_{z})_{V^{\perp}}|^{2} + \beta|\nu_{x}+\nu_{z}|^{2}  \} \epsilon^{2} \\ &
+ (4M +10  |x-z|^{\gamma -2})\epsilon^{2}.
\end{align*} 
Set $$ \kappa = \frac{ |(\nu_{x}-\nu_{z})_{V^{\perp}}|^{2}}{\Theta} .$$ Then $ 1 < \kappa \le \frac{4}{\Theta} $ by the assumption. Observe that $$|(\nu_{x}-\nu_{z})_{V}| \le  \sqrt{|\nu_{x}-\nu_{z}|^{2}-\kappa \Theta} \le \sqrt{4-\kappa \Theta} \le 2 - \frac{\kappa}{4}\Theta$$ and hence $$ |(\nu_{x}-\nu_{z})_{V^{\perp}}| \le 4( 2 - (\nu_{x}-\nu_{z})_{V}).$$ 
On the other hand, we have 
\begin{align*} |\nu_{x}+\nu_{z}|^{2} & = 4 - |\nu_{x}-\nu_{z}|^{2}
\\ & \le 4 - (\nu_{x}-\nu_{z})_{V}^{2}
\\ & \le 4 ( 2- (\nu_{x}-\nu_{z})_{V}).
\end{align*}
We observe that
\begin{align*} \frac{1}{2}& \big\{ Tf_{1} ( x, z,   P_{\nu_{x}}, P_{\nu_{z}} ) + Tf_{1} (  x, z,P_{-\mathbf{v}},  P_{\mathbf{v}} ) \big\}  - f_{1}(x,z) \\ &
\le   2\alpha M\langle x+z , \nu_{x}+\nu_{z} \rangle  \epsilon+ (2 - (\nu_{x}-\nu_{z})_{V} )  C \omega ' {(|x-z|)} \bigg( - \alpha  +
\frac{20}{N} \bigg) \epsilon \\ & \qquad
+ (4M +10  |x-z|^{\gamma -2})\epsilon^{2} ,
\end{align*} as $|x-z| > N\epsilon / 10$.

Next we estimate $  M\langle x+z , \nu_{x}+\nu_{z} \rangle \epsilon$.
We already know that $u_{\epsilon} $ satisfies H\"{o}lder type estimate for any exponent $ \delta \in (0,1) $ by Theorem \ref{thm1}.
Now by the counter assumption (\ref{counter}), 
$$ u_{\epsilon}(x,t)-u_{\epsilon}(z,s)-C \omega(|x-z|) - M|x+z|^{2}- g(t,s) \ge K - \eta > 0.$$ 
Then we see 
\begin{align*} M|x+z|^{2}    <  u_{\epsilon}(x,t)-u_{\epsilon}(z,s)  \le C_{u_{\epsilon}} (|x-z|^{1/2}+ \epsilon^{1/2})
.
\end{align*}
Note that $  C_{u_{\epsilon}}$ is a constant depending only on $r, \alpha$ and $n$.
Thus, we obtain that
\begin{align*} |x+z| & < \sqrt{ \frac{ C_{u_{\epsilon}}}{M}} (|x-z|^{1/2}+ \epsilon^{1/2})^{1/2} \\& \le  \sqrt{ \frac{ C_{u_{\epsilon}}}{M}} \bigg[ |x-z|^{1/4}+ \frac{1}{2}|x-z|^{-1/4}\epsilon^{1/2} + o (\epsilon^{1/2})\bigg]
\\& \le  \sqrt{ \frac{ C_{u_{\epsilon}}}{M}} \bigg[ |x-z|^{1/4}+ \frac{1}{2}\bigg( \frac{10}{N} \bigg)^{1/4}\epsilon^{1/4} + o (\epsilon^{1/2})\bigg].
\end{align*}
Hence we observe
\begin{align*}M\langle x+z  &, \nu_{x}+\nu_{z} \rangle \epsilon  \le 2M|x+z|\epsilon \\ &
\le 2 \sqrt{MC_{u_{\epsilon}}} |x-z|^{1/4}\epsilon+ \sqrt{MC_{u_{\epsilon}}} \bigg( \frac{10}{N} \bigg)^{1/4}\epsilon^{5/4} + o (\epsilon^{3/2}) \\&
\le 3 \sqrt{MC_{u_{\epsilon}}} |x-z|^{1/4}\epsilon 
\end{align*}
since $ \sqrt{MC_{u_{\epsilon}}} (10/N )^{1/4}\epsilon^{5/4} + o (\epsilon^{3/2}) $ is bounded by $\sqrt{MC_{u_{\epsilon}}} |x-z|^{1/2}\epsilon$.
Therefore, if we choose $ \gamma = 1+ s = 5/4 $, \begin{align*} & \frac{1}{2}   \big\{ Tf_{1} ( x, z, P_{\nu_{x}}, P_{\nu_{z}} ) + Tf_{1} (  x, z,P_{-\mathbf{v}},  P_{\mathbf{v}} ) \big\} - f_{1}(x,z) \\ &
\le   6\alpha  \sqrt{MC_{u_{\epsilon}}} |x-z|^{s}   \epsilon + C \omega ' {(|x-z|)}  \times \\ & \qquad \ \ \ \bigg[ - \alpha \kappa \frac{ |x-z|^{s} }{4} +
\frac{5}{N} \bigg\{ \alpha |(\nu_{x}-\nu_{z})_{V^{\perp}}|^{2} + \frac{\beta}{n+1}|(\rho_{x} - \rho_{z} )_{V^{\perp}}|^{2}  \bigg\} \bigg] \epsilon \\ &
+ (4M +10  |x-z|^{s-1})\epsilon^{2}.
\end{align*}
Note that $$ (4M +10  |x-z|^{s-1})\epsilon^{2} \le (4M+10) \frac{10}{N} |x-z|^{s} \epsilon. $$ Then  
\begin{align*}  \frac{1}{2}   \big\{ & Tf_{1} ( x, z, P_{\nu_{x}}, P_{\nu_{z}} ) + Tf_{1} (  x, z, P_{-\mathbf{v}},  P_{\mathbf{v}}) \big\}  - f_{1}(x,z)  \\ &
\le  6\alpha  \sqrt{MC_{u_{\epsilon}}} |x-z|^{s} \epsilon  + (2 - (\nu_{x}-\nu_{z})_{V} ) C \omega ' {(|x-z|)}    \bigg( -  \alpha +
\frac{20}{N} \bigg) \epsilon  \\ & \qquad 
+ (4M+10) \frac{10}{N}  |x-z|^{s} \epsilon.
\end{align*} 

Since we already know that $ \kappa \Theta / 4 \le 2 -(\nu_{x}-\nu_{z})_{V}  $ and $ \omega ' \in [1/2, 1] $, we see that
\begin{align*}  \frac{1}{2}   \big\{ & Tf_{1} ( x, z, P_{\nu_{x}}, P_{\nu_{z}} ) + Tf_{1} (  x, z,  P_{-\mathbf{v}},  P_{\mathbf{v}} ) \big\}  - f_{1}(x,z)  \\ &
\le  \bigg[ 6\alpha  \sqrt{MC_{u_{\epsilon}}}  + C    \bigg( -  \frac{ \alpha }{8} +
\frac{5}{N}  \bigg) \frac{ |(\nu_{x}-\nu_{z})_{V^{\perp}}|^{2}}{|x-z|^{s}}  
+ (4M+10) \frac{10}{N} \bigg] |x-z|^{s} \epsilon \\ &
\le \bigg[ 6\alpha  \sqrt{MC_{u_{\epsilon}}}  + C    \bigg( -  \frac{ \alpha }{8} +
\frac{5}{N}  \bigg)   
+ (4M+10) \frac{10}{N} \bigg] |x-z|^{s} \epsilon .
\end{align*}
Fix $N > 100/\alpha$ and choose $C = C( r, \alpha, n)$ large enough so that 
$$  6\alpha  \sqrt{MC_{u_{\epsilon}}}  + C    \bigg( -  \frac{ \alpha }{8} +
\frac{5}{N}  \bigg)   
+ (4M+10) \frac{10}{N} < 0 .$$
Then we conclude that
\begin{align*}  \frac{1}{2}   \big\{ & Tf_{1} ( x, z,  P_{\nu_{x}}, P_{\nu_{z}} ) + Tf_{1} (  x, z, P_{-\mathbf{v}},  P_{\mathbf{v}} ) \big\}  - f_{1}(x,z)  \\ &
\le \frac{N}{10} \bigg[ 6\alpha  \sqrt{MC_{u_{\epsilon}}}  + C    \bigg( -  \frac{ \alpha }{8} +
\frac{5}{N}  \bigg)   
+ (4M+10) \frac{10}{N} \bigg] |x-z|^{s-1} \epsilon^{2} \\ &
\le -M \tilde{C} \epsilon^{2} ,
\end{align*} since $ |x-z| > N \epsilon /10$. Now we obtained the desired result.

\subsection{Case \texorpdfstring{$0<|x-z| \le N \epsilon / 10$}.} It is quite similar to the H\"{o}lder case. First, we see that for any $x , z \in B_{r}$ and $h_{x}, h_{z} \in S^{n-1}$,
\begin{align*} & |f_{1}(x+\epsilon h_{x},z+\epsilon  h_{z})-f_{1}(x,z)| \\&
\le C \big|\omega(|x+\epsilon h_{x}-z-\epsilon h_{z}|)-\omega(|x-z|) \big| \\ & \qquad + M\big| |x+\epsilon h_{x}+z+\epsilon h_{z}|^{2}-|x+z|^{2} \big| \\&
\le C \big( \big||x+\epsilon h_{x}-z-\epsilon h_{z}|- |x-z| \big|+ \omega_{0}\big| |x+\epsilon h_{x}-z-\epsilon h_{z}|^{\gamma} -|x-z |^{\gamma} \big| \big) \\ &
 \qquad +M\big| |x+\epsilon h_{x}+z+\epsilon h_{z}|^{2}-|x+z|^{2} \big| \\&
\le 2C\epsilon + 2C \omega_{0} \gamma (2r)^{\gamma-1}(2\epsilon) + 8Mr\epsilon+ 4M\epsilon^{2}. 
\end{align*}
Then we can choose a constant $C >0$ such that
$$|f_{1}(x+\epsilon h_{x},z+\epsilon  h_{z})-f_{1}(x,z)| \le 20C \epsilon.  $$ 
As in the previous section,
\begin{align*} &\sup_{h_{x},h_{z} \in  S^{n-1} }  Tf_{1}( x, z,P_{h_{x}}, P_{h_{z}}  ) - f_{1}(x,z) \\& \qquad   =  \sup_{h_{x},h_{z} \in  S^{n-1} } \bigg[ \alpha\{ f_{1}(x+ \epsilon h_{x},z+ \epsilon h_{z})-f_{1}(x,z)   \} \\& \qquad \qquad \qquad \qquad  + \beta \kint_{B_{\epsilon}^{e_{1}} }\{f_{1}(x+ P_{h_{x}}h,z+P_{h_{z}} h) - f_{1}(x,z) \}  d \mathcal{L}^{n-1}(h) \bigg] \\& \qquad
\le 20\alpha C\epsilon+ 20\beta C\epsilon= 20C\epsilon
\end{align*} and note that (\ref{anes}) is still valid here.
We can find $i \in \{1, 2, \cdots , N\}$ such that  $ (i-1) \frac{ \epsilon}{10} < |x-z| \le i\frac{ \epsilon}{10} $ as in the previous section. 
Now, if $C$ is large enough,
\begin{align*} &\inf_{h_{x},h_{z} \in  S^{n-1}} Tf(  x, z, P_{h_{x}}, P_{h_{z}}   ) \\ & \le \sup_{h_{x},h_{z} \in  S^{n-1}} Tf_{1}(  x, z, P_{h_{x}}, P_{h_{z}}   )  - \sup_{h_{x},h_{z} \in  S^{n-1}} T f_{2}(  x, z, P_{h_{x}}, P_{h_{z}}   ) \\&
\le  \sup_{h_{x},h_{z} \in  S^{n-1}} T f_{1}(  x, z,P_{h_{x}}, P_{h_{z}}  ) - \alpha C^{2(N-i+1)} \epsilon\\&
= \sup_{h_{x},h_{z} \in  S^{n-1}} T f_{1}( x, z,P_{h_{x}}, P_{h_{z}}   ) - \alpha \bigg( C^{2}- \frac{2}{\alpha} \bigg) C^{2(N-i)} \epsilon - 2C^{2(N-i)} \epsilon \\&
\le \sup_{h_{x},h_{z} \in  S^{n-1}} T f_{1}(  x, z,P_{h_{x}}, P_{h_{z}}   ) - 2f_{2}(x,z) - 50C \epsilon .
\end{align*}
 Therefore, we calculate that
\begin{align*}  & \midrg_{h_{x},h_{z} \in  S^{n-1}}  Tf(  x, z,P_{h_{x}}, P_{h_{z}}   ) \\&  \qquad \le  \sup_{h_{x},h_{z} \in  S^{n-1}} T f_{1}( x, z, P_{h_{x}}, P_{h_{z}}   ) - f_{2}(x,z)-25C\epsilon
\\&  \qquad  \le  f_{1}(x,z) + 20C\epsilon -  f_{2}(x,z)-25C\epsilon.
\end{align*}
We finally choose a large constant $C > M$ depending only on $r, \alpha$ and $n$ to obtain (\ref{eq:hfirst}).

\subsection{Case \texorpdfstring{$|x-z| =0$}.} Similar to the previous section, we already showed that 
$$ |u_{\epsilon} (x,t) - u_{\epsilon} (z,s) | \le C_{2} ||u_{\epsilon}||_{\infty} ( |x-z| + \epsilon) ,$$ for any $x, z \ (x \neq z) \in B_{r}(0)$, $ -r^{2}<t<0$, $|t-s| < \epsilon^{2} / 2 $ and some $C_{2}= C_{2}(r, \alpha , n)>0$.
Then we can obtain the desired result by using the same argument as in Section 4.3.
\end{proof}

Now Lemma \ref{lem2} and Lemma {\ref{lem3}} yield the Lipschitz type regularity in the whole cylinder. 

\begin{theorem} \label{mainthm}
Let $\bar{Q}_{2r} \subset \Omega_{T}$, $0 < \alpha <1$ and $\epsilon > 0 $ be small. 
Suppose that $u_{\epsilon}$ satisfies the $\alpha$-parabolic DPP with boundary data $F \in L^{\infty}(\Gamma_{\epsilon, T} )$.
Then for any $x, z \in B_{r}(0)$ and $ -r^{2} < t, s< 0  $, 
$$ |u_{\epsilon} (x,t) - u_{\epsilon} (z,s) | \le C ||u_{\epsilon}||_{\infty} ( |x-z|+ |s-t|^{\frac{1}{2}} + \epsilon), $$ 
where $C>0$ is a constant which only depends on $r, \alpha$ and $n$.
\end{theorem}

{\bf Acknowledgments}.
This work was supported by NRF-2015R1A4A1041675.
The author would like to thank M. Parviainen, for introducing this topic, valuable discussions and constant support throughout this work. 
The work was partly done while the author was visiting Univerity of Jyv\"{a}skyl\"{a} (JYU) in Finland and thanks JYU for the hospitality.

\bibliographystyle{alpha}

\end{document}